\documentclass{amsart}
\usepackage{amsmath}
\usepackage{amssymb}

\usepackage{enumitem,kantlipsum}
\usepackage[margin=1in]{geometry}
\usepackage[all,cmtip]{xy}
\usepackage{color,graphicx}
\usepackage{amsthm}
\usepackage{dsfont}
\usepackage{comment}
\usepackage{mathrsfs}
\usepackage[toc,page]{appendix}
\usepackage[utf8]{inputenc}
\usepackage{tikz, subfigure}
\usepackage[mathscr]{euscript}

\usepackage[bookmarks=true, bookmarksopen=true,%
bookmarksdepth=3,bookmarksopenlevel=2,%
colorlinks=true,%
linkcolor=blue,%
citecolor=blue,%
filecolor=blue,%
menucolor=blue,%
urlcolor=blue]{hyperref}
\usepackage{cleveref}
\usepackage{autonum}

\def\CC{\mathbb{C}}

\def\EE{\mathbb{E}}
\def\FF{\mathbb{F}}
\def\GG{\mathbb{G}}

\def\LL{\mathbb{L}}

\def\QQ{\mathbb{Q}}
\def\RR{\mathbb{R}}

\def\ZZ{\mathbb{Z}}

\def\calC{\mathscr{C}}

\def\calL{\mathcal{L}}

\def\calN{\mathcal{N}}

\newcommand\frC{\mathfrak{C}}
\newcommand\frD{\mathfrak{D}}

\newcommand\frF{\mathfrak{F}}
\newcommand\frG{\mathfrak{G}}

\newcommand\frK{\mathfrak{K}}

\newcommand\frS{\mathfrak{S}}

\newcommand\frb{\mathfrak{b}}

\newcommand\frg{\mathfrak{g}}

\newcommand\frl{\mathfrak{l}}
\newcommand\frakm{\mathfrak{m}}

\newcommand\frp{\mathfrak{p}}
\newcommand\frq{\mathfrak{q}}

\newcommand\frt{\mathfrak{t}}

\newcommand\frz{\mathfrak{z}}

\newcommand\tilW{\widetilde{W}}

\newcommand\scrB{\mathscr{B}}
\newcommand\scrC{\mathscr{C}}
\newcommand\scrD{\mathscr{D}}

\newcommand{\Coh}{\textup{Coh}}

\newcommand{\gr}{\textup{gr}}

\newcommand\IC{\textup{IC}}

\newcommand{\Ind}{\textup{Ind}}
\newcommand{\ind}{\textup{ind}}

\newcommand{\Perf}{\textup{Perf}}

\newcommand{\QCoh}{\textup{QCoh}}

\newcommand{\reg}{\textup{reg}}

\newcommand{\Res}{\textup{Res}}
\newcommand\res{\textup{res}}

\newcommand\Spec{\textup{Spec}}

\newcommand\Sym{\textup{Sym}}

\newcommand{\Vect}{\textup{Vect}}

\newcommand{\Stack}{\textup{Stack}}

\newcommand{\Spr}{\textup{Spr}}

\newcommand\Hom{\textup{Hom}}
\newcommand\End{\textup{End}}

\newcommand\RHom{\textup{RHom}}

\newcommand\REnd{\textup{REnd}}

\newcommand\nc{\newcommand}
\nc\on{\operatorname}
\nc\ol{\overline}
\nc\ul{\underline}
\nc\us[1]{\underline{\smash{#1}}}

\nc\oo{\infty}

\nc\Cone{\mathit{Cone}}
\nc\ssupp{\mathit{ss}}
\nc\risom{\stackrel{\sim}{\to}}
\nc\Sh{\textup{Sh}}
\nc\un{\diamondsuit}
\nc\orient{\mathit{or}}
\nc\sing{\mathit{sing}}
\nc\MF{\on{MF}}
\nc\inthom{\mathit{Hom}}

\newcommand{\Ch}{\textup{Ch}}

\newcommand{\C}{\mathbb{C}}

\newcommand{\modu}{\textup{-mod}}

\newcommand{\colim}{\textup{colim}}

\newcommand{\beq}{\begin{equation}}
	\newcommand{\eeq}{\end{equation}}

\newcommand{\scrCh}{\mathscr{C}\textup{h}}
\newcommand{\scrQCoh}{\mathscr{QC}\textup{oh}}
\newcommand{\scrCoh}{\mathscr{C}\textup{oh}}
\newcommand{\scrPerf}{\mathscr{P}\textup{erf}}

\newcommand{\Pur}{\textup{Pur}}
\newcommand{\modfd}{\textup{-mod}_{\textup{fd}}}
\newcommand{\perf}{\textup{-perf}}
\newcommand{\open}{\textup{open}}

\newcommand{\prin}{\textup{prin}}

\newcommand{\hst}{{\heartsuit_t}}
\newcommand{\hsw}{{\heartsuit_w}}
\newcommand{\hstw}{{\heartsuit_{t,w}}}
\newcommand{\oblv}{\textup{oblv}}

\newcommand{\sfCh}{\mathsf{Ch}}
\newcommand{\sfQCoh}{\mathsf{QCoh}}
\newcommand{\sfCoh}{\mathsf{Coh}}
\newcommand{\sfPerf}{\mathsf{Perf}}
\newcommand{\sfA}{\mathsf{A}}
\newcommand{\sfB}{\mathsf{B}}
\newcommand{\sfRes}{\mathsf{Res}}
\newcommand{\sfRef}{\mathsf{Ref}}
\newcommand{\sfU}{\mathsf{U}}

\newcommand{\Cat}{\mathscr{C}\textup{at}}
\newcommand{\TCat}{\mathscr{TC}\textup{at}}
\newcommand{\WCat}{\mathscr{WC}\textup{at}}

\newtheorem{thm}[equation]{Theorem}
\newtheorem{prop}[equation]{Proposition}
\newtheorem{lem}[equation]{Lemma}
\newtheorem{cor}[equation]{Corollary}

\theoremstyle{definition}
\newtheorem{defn}[equation]{Definition}

\newtheorem{eg}[equation]{Example}
\newtheorem{no}[equation]{Notation}
\newtheorem{rmk}[equation]{Remark}

\numberwithin{equation}{subsection}

\def\tilfrC{\widetilde{\mathfrak{C}}}
\def\tilfrS{\widetilde{\mathfrak{S}}}

\newcommand\ov{\overline}

\begin{document}

\nc{\Conn}{\mathrm{Conn}}

\nc{\fg}{\mathfrak g}
\nc{\fh}{\mathfrak h}

\nc{\cN}{\mathcal N}

%%%%%%%%%%%%%%%%%%%%%%%%%%%%%%%%%%%%%%%%%%%%%%%%%%%%%%%%%
%%%%%%%%%%%%%%%%%%%%%%%%%%%%%%%%%%%%%%%%%%%%%%%%%%%%%%%%%
%%%%%%%%%%%%%%%%%%%%%%%%%%%%%%%%%%%%%%%%%%%%%%%%%%%%%%%%%

\title[Derived categories of character sheaves II]{Derived categories of character sheaves II: \\ canonical induction/restriction functors}

\author{Penghui Li}

\address{Yau Mathematical Science Center, Tsinghua University}
\email{lipenghui@mail.tsinghua.edu.cn}

\begin{abstract}
We give a combinatorial description of the dg category of character sheaves on a complex reductive group $G$, extending results of \cite{liDerivedCategoriesCharactera} for $G$ simply-connected. We also explicitly identify the parabolic induction/restriction functors. %As a input, we also explain that 
\end{abstract}

\dedicatory{to Lanxin.}

\maketitle

%%%%%%%%%%%%%%%%%%%%%%%%%%%%%%%%%%%%%%%%%%%%%%%%%%%%%%%%%
%%%%%%%%%%%%%%%%%%%%%%%%%%%%%%%%%%%%%%%%%%%%%%%%%%%%%%%%%
%%%%%%%%%%%%%%%%%%%%%%%%%%%%%%%%%%%%%%%%%%%%%%%%%%%%%%%%%

\section{Introduction}

\subsection{Main results}

Let $G$ be a connected complex reductive group. Denote $\Sh(G/G)$ the dg-category of all conjugate equivariant sheaves (in classical topology) on $G$.  The category of character sheaves on $G$ is by definition $\Ch(G/G) \subset \Sh(G/G) $ the full subcategory consists of sheaves with nilpotent singular support, and the category of constructible character sheaves $\Ch_c(G/G) \subset \Ch(G/G)$ is the full subcategory consists of constructible sheaves. This note gives a combinatorial description of these categories. 

Fix $T \subset G$ a maximal torus, denote $X_*(T)$ the coweight lattice,  $W_G=N_G(T)/T$ the Weyl group, and $\widetilde{W}_G = W_G \ltimes X_*(T)$ the extended affine Weyl group. 
%Let $\widetilde{\frL}$ be the set of Levihorics of the loop group $LG$ which contains $T$. 
Let $\widetilde{\frC}_G$ be the set of triples $(L,R,F)$ where $L$ is a Levihoric subgroup of the loop group $LG$, $R \subset L$ is a Borel subgroup which contains $T$, and $F$ is an isomorphism class of cuspidal sheaves on the nilpotent cone $\calN_{L}/L$. For $c=(L,R,F) \in \widetilde{\frC}_G$, we write $\frz_c =\frz_{\frl}$ the center of Lie algebra $\frl$ of $L$. Note $\widetilde{W}_G$ acts naturally on the set $\widetilde{\frC}_G$.  Our main theorem reads:
\begin{thm}[Theorem~\ref{thm:main}]
	\label{intro:thm_main}
	There is an equivalence of dg-categories:
	$$\Ch(G/G) \simeq (\prod_{c \in \widetilde{\frC}_G}  \Sym(\frz_c[1]) \textup{-mod})^{\widetilde{W}_G};  $$
	$$ \Ch_c(G/G) \simeq (\prod_{c \in \widetilde{\frC}_G}  \Sym(\frz_c[1]) \textup{-mod}_{\textup{fd}})^{\widetilde{W}_G}.    $$
	where we denote $\Sym$ the (differential-graded) symmetric algebra,  \textup{mod} the dg-category of all modules, and $\textup{mod}_{\textup{fd}}$ the dg-category of finite dimensional modules. The superscript ${\widetilde{W}}_G $  denotes the invariant category of the natural action induced by ${\widetilde{W}}_G $-action on $\widetilde{\frC}_G$.
\end{thm}

Moreover, let $L \subset G$ be a Levi subgroup of some parabolic $P$, which contains $T$. Then the natural inclusion $\widetilde{W}_L \subset \widetilde{W}_G$ and $\widetilde{\frC}_L \subset \widetilde{\frC}_G $ induces
a functor via restriction:
$$\Res:  (\prod_{c \in \widetilde{\frC}_G}  \Sym(\frz_c[1]) \textup{-mod})^{\widetilde{W}_G} \to   (\prod_{c \in \widetilde{\frC}_L}  \Sym(\frz_c[1]) \textup{-mod})^{\widetilde{W}_L}.$$

Denote by $\Ind$ the left adjoint of $\Res$. On the other hand, we have a pair of adjoint functors by parabolic induction and restriction:
$$  \Ind_{L \subset P}^G : \Ch(L/L) \rightleftarrows \Ch(G/G) : \Res_{L \subset P}^G $$

\begin{thm}[Theorem~\ref{thm:main}]
	\label{intro:prop_parabolic_restriction}
	Under the equivalence in Theorem~\ref{intro:thm_main},
the diagrams naturally commute:
$$\xymatrix{
	\Ch(L/L)  \ar@<-.5ex>[d]_{\Ind_{L \subset P}^G}  \ar[r]^-{\simeq}  &    (\prod_{c \in \widetilde{\frC}_L}  \Sym(\frz_c[1]) \textup{-mod})^{\widetilde{W}_L}  \ar@<-.5ex>[d]_{\Ind} \\
	\Ch(G/G)  \ar[r]^-{\simeq} \ar@<-.5ex>[u]_{\Res_{L\subset P}^G}  &  (\prod_{c \in \widetilde{\frC}_G}  \Sym(\frz_c[1]) \textup{-mod})^{\widetilde{W}_G} \ar@<-.5ex>[u]_{\Res} 
}$$
Similar statement holds for the category of constructible character sheaves as well.
\end{thm}

\begin{rmk}
When $G$ is simply-connected, we have natural bijection between sets $\tilfrC_{G}/{\widetilde{W}_G} \simeq \{(J,F)\}$, where $J$ is  a facet of a fixed affine chamber, and $F$ a cuspidal sheaf on the nilpotent cone of the Levihoric associated to $J$. Thus in view of Remark~\ref{rmk:invariant} below,  Theorem~\ref{intro:thm_main} recovers the main theorem of \cite{liDerivedCategoriesCharactera}.
\end{rmk}

\subsection{Ideas and strategies} We give an informal introduction of the main ideas. The discussion in this section will not be used later. The character sheaves form a sheaf of categories on $G/G$, which we denote by $\scrCh_G$. 
Let $\chi:G/G \to T//W_G$ be the characteristic polynomial map. Then $\chi_*(\scrCh_G)$ is a sheaf of categories on $T//W_G$. On the other hand, we explicitly define a sheaf of category $\scrQCoh_{{\tilfrC}_G}$ on the affine space $\frt$. 
Roughly speaking, $\scrQCoh_{{\tilfrC}_G}$ is a product of constant sheaves of categories on affine root subspaces of $\frt$, with multiplicity labeled by ${\tilfrC}_G$. 
The sheaf of categories $\scrQCoh_{{\frC}_G}$ is naturally $\tilW_G$-equivariant, therefore we may identify it as a sheaf of categories on the quotient stack $\frt/\tilW_G$, denote by $\scrQCoh_{{\frC}_G}^{\tilW_G}$. 
Let $\pi: \frt/\tilW_G \to \frt//\tilW_G \simeq T//W_G$ be the natural projection that forgets the stabilizer. We expect that there is an isomorphism of sheaves of categories over $T//W_G$: 
\begin{equation}
	\label{eq:Ch=QCoh_as_sheaves}
 \chi_*(\scrCh_G)  \simeq         \pi_*(\scrQCoh_{{\frC}_G}^{\tilW_G})      
\end{equation}

Theorem~\ref{intro:thm_main} can be recovered by taking global sections. At the level of stalk, (\ref{eq:Ch=QCoh_as_sheaves}) is easy to check: 
indeed for any $s \in \frt$, and $\bar{s}$ the image of $s$ in $T//W$, 
one can define an reductive subgroup $G_{\bar{s}}$ of the loop group $LG$ (which is the stabilizer of $s$ for the twisted adjoint action of $LG$).  
We can identify the stalk $\chi_*(\scrCh_G)_{\bar{s}}$ with $\Ch(\frg_{s}/G_{s})$ the category of character sheaves on the Lie algebra $\frg_ {s}$. 
And can be further identifed with the stalk  $\pi_*(\scrQCoh_{{\frC}_G}^{\tilW_G})_{\bar{s}}$ by the explicity description of character sheaves on Lie algebra \cite{riderFormalityLusztigGeneralized2021,liDerivedCategoriesCharactera}.

In this note, we shall not directly show this equivalence as sheaves of categories. Instead, we pick an explicit cover that computes the global sections. To do this properly in the context of $\infty$-categories, we shall consider the graded lift/mixed version of character sheaves categories on Lie algebra, 
as done in \cite{riderFormalityLusztigGeneralized2021}. 
 We reinterpret the graded lift using categories with transversal weight and $t$-structure (See Section~\ref{sec:cat_pre} and appendix~\ref{secA:graded_lift} for facts about the category theory we use.).
The advantage of graded lift is that pure objects form an $1$-category $\Pur(X)$, and therefore the infinitely many coherent data in a diagram of $\infty$-categories reduces to the commutative diagram of $1$-categories $\Pur(X)$'s.

A subtle point is that the local identifications \textit{a priori} depends on choices of parabolic subgroups. Instead of keeping track of this choice as in \cite{liDerivedCategoriesCharactera}, we shall show that the functors involved (which in particular include parabolic induction functors) are canonically independent of the choice of parabolics. This is carried out in Section~\ref{sec:canonical_functors}.

Finally, in Section~\ref{sec:gluing} we glue the local identifications to obtain our main theorem. On the character sheaves side, this is essentially \cite{liUniformizationSemistableBundles2021}, with the combinatorics slightly modified. On the spectral side, this is a straightforward computation.

\subsection{Acknowledgements}

This note is motivated by the joint projects with Quoc P. Ho, and independently with David Nadler and Zhiwei Yun. We thank them for their interest and inspiration. The author was partially supported by the National Natural Science Foundation of China (Grant No. 12101348).

\section{Preliminary on category theory}
\label{sec:cat_pre}

\subsection{$2$-categories}
We shall briefly recall the notion of $2$-categories as defined in \cite{johnson2DimensionalCategories2020}. 

\begin{defn}
	A \textit{$2$-category} $\mathscr{B}$ contains the following data: 
	\begin{itemize}
		\item A class of objects $\mathscr{B}_0$.
		\item A class $\mathscr{B}_1(X,Y)$ of $1$-cells from $X$ to $Y$, for $X,Y \in \mathscr{B}_0$.
		\item An identity $1$-cell $1_X \in \scrB_1(X,X), $ for $X \in \scrB_0$.
		\item A set $\scrB_2(X,Y)(f,f')$, or simply $\Hom_{\scrB}(f,f')$ of $2$-cells from $f$ to $f'$, for $1$-cell $f,f' \in \scrB_1(X,Y)$.
		\item An identity $2$-cell $1_f \in \Hom_{\scrB}(X,Y)$, for each $1$-cell $f \in \scrB_1(X,Y)$ and each pair of object $X,Y$.
		\item Vertical composition: 
		 $$ \Hom_{\scrB}(f',f'') \times \Hom_{\scrB}(f,f')  \longrightarrow \Hom_{\scrB}(f,f''), 
		 \qquad  (\alpha',\alpha) \mapsto \alpha'\alpha  ,$$
		 for objects $X,Y$ , and $1$-cells $f,f',f'' \in \scrB_1(X,Y)$.
		 \item Horizontal composition of $1$-cells:
		 $$ \scrB_1(Y,Z) \times \scrB_1(X,Y) \longrightarrow \scrB_1(X,Z), \qquad (g,f) \mapsto gf , $$
		 for objects $X,Y$ and $Z$.
		 \item Horizontal composition of $2$-cells:
		 $$ \Hom_{\scrB}(g,g') \times \Hom_{\scrB}(f,f') \longrightarrow \Hom_{\scrB}(gf,g'f'), 
		 \qquad  (\beta, \alpha) \mapsto \beta \star \alpha. $$
	\end{itemize}
These data are required to satisfy the axioms as in \cite[Proposition 2.3.4]{johnson2DimensionalCategories2020}.

By a $(2,1)$-\textit{category}, we mean a 2-category whose 2-cells are invertible.
\end{defn}

\begin{eg}
	The collection of small $1$-categories forms a $2$-category, with $1$-cell functor between $1$-categories, and $2$-cell natural transformation between functors. We denote this $2$-category by $\Cat_1$.
\end{eg}

\begin{eg}
A commutative square in a $2$-category $\scrB$:
$$\xymatrix{ X \ar[r]^f  \ar@{}[dr] | {\alpha} \ar[d]_g &  Y  \ar[d]^h  \\ 
			Z  \ar[r]_k  &    W }$$
		is the data of an \textit{invertible} $2$-cell $\alpha \in \Hom_{\scrB}(kg,hf)$. 
\end{eg}

We refer to \cite[Definition 4.1.2]{johnson2DimensionalCategories2020} for the definition of functors between $2$-categories. In particular, a functor $F: \scrB \to \scrB'$ consists: 
\begin{itemize}
	\item  A function $F:\scrB_0 \to \scrB_0'$ on objects.
	\item  A functor $F: \scrB_1(X,Y) \to \scrB'_1(X,Y)$, for each pair of object $X,Y$.
\end{itemize}

\begin{defn}
	A functor $F: \scrB \to \scrB'$ is \textit{locally faithful} if the functor $F: \scrB_1(X,Y) \to \scrB'_1(X,Y)$ is faithful, for all object $X,Y$ in $\scrB$. More explicitly, this means for any $2$-cells $f,g \in \scrB_1(X,Y)$, the induced map between sets
	$$ \Hom_{\scrB}(f,g) \longrightarrow  \Hom_{\scrB'}(Ff,Fg)  $$ 
	is injective.
\end{defn}

\subsection{Transversal categories} 
We refer to \cite{lurieHigherToposTheory2009a, lurieHigherAlgebra2012} for the theory of $\infty$-categories. By \textit{dg-category}, we mean a $\CC$-linear stable $\infty$-category. 
The notion of weight structure is developed by \cite{pauksztelloNoteCompactlyGenerated2011,bondarkoWeightStructuresVs2010}, and the transversal $t$-structure by \cite{bondarkoWeightStructuresWeights2012}. Readers can also refer to \cite[Section 5]{hoRevisitingMixedGeometry2022} for some basic facts. In this section, categories are assumed to be small unless otherwise specified.

\begin{defn} A \textit{transversal} $\infty$-category is a stable $\infty$-category $\mathscr{C}$ equipped with a bounded weight structure and a bounded transversal $t$-structure, whose weight heart is equivalent to an 1-category. 
\end{defn}

For a transversal category $\mathscr{C}$, we denote by $\mathscr{C}^{\heartsuit_{w}}$, $\mathscr{C}^{\heartsuit_{t}}$ the weight and $t$-heart. Put  $\mathscr{C}^{\heartsuit_{w,t}}= \mathscr{C}^{\heartsuit_{w}} \cap \mathscr{C}^{\heartsuit_{t}}$. It is known that the weight structure and $t$-structure of a transversal category can be recovered from the additive 1-subcategories $\{\scrC^{w=k} \cap \scrC^{\hst}\}_{k \in \ZZ}$ (which forms a semi-orthogonal family).

\begin{eg} Denote by $\Vect^c $ (resp. $(\Vect^{\gr,c})$ )
	the dg category of finite dimensional vector spaces (resp. graded vector spaces). 
	Then $\Vect^{\gr,c}$ is a transversal category, with $\Vect^{\gr,c,\heartsuit_{w}}= \{ \oplus V_n : H^i(V_n) =0, \forall i\neq n   \} $, and $\Vect^{\gr,c,\heartsuit_{t}}= \{ \oplus V_n : H^i(V_n) =0, \forall i\neq 0   \} $.
\end{eg}

\begin{defn} A graded \textit{transversal} $\infty$-category is an transversal $\infty$-category $\mathscr{C}$ together with an action of $\Vect^{\gr,c}$, such that the action of $\Vect^{\gr,c}$ on $\mathscr{C}$ is both weight and $t$-exact.
\end{defn}

\begin{eg}[\cite{hoRevisitingMixedGeometry2022} Propsition 5.6.12]
	\label{eg:graded_sheaves}
	The graded lift $\Sh_{\gr,c,S}(X)$ as in Appendix~\ref{secA:graded_lift} is a graded transversal category, the weight heart is given by $\Pur_{c,S}(X)$, an object in $\Sh_{\gr,c,S}(X)$ is in the $t$-heart if and only if its image in $\Sh_{c}(X)$ is perverse. Moreover $\Sh_{\gr,c,S}(X)^{\heartsuit_{w,t}}$ is given by the additive category of semisimple perserve sheaves on $X$ generated by $S$.
\end{eg}

\begin{no}
\begin{enumerate}
	\item For $\scrD$ a dg-category, and $c,d \in \scrD$, we denote 
	\begin{itemize}
		\item  $\RHom(c,d)$ the $\Hom$-complex in $\scrD$.
		\item $ \Hom^{i}(c,d)= H^i(\RHom(c,d))$, and $\Hom^{*}(c,d)= \oplus_{i \in \ZZ} \Hom^{i}(c,d)$.
		\item  $\REnd(c)=\RHom(c,c), \End^{i}(c)=\Hom^{i}(c,c)$ and $\End^{*}(c)=\Hom^{*}(c,c)$.
	\end{itemize}
		\item For $\scrC$ a $\Vect^{\gr,c}$ linear $\infty$-category, and $c,d \in \scrC$, we denote 
	\begin{itemize}
		\item  $\RHom^\gr(c,d)$ the graded $\Hom$-complex in $\scrC$.
		\item $ \Hom^{i,\gr}(c,d)= H^i(\RHom^\gr(c,d))$, and $\Hom^{*,\gr}(c,d)= \oplus_{i \in \ZZ} \Hom^{i,\gr}(c,d)$.
		\item  $\REnd^\gr(c)=\RHom^\gr(c,c), \End^{i,\gr}(c)=\Hom^{i,\gr}(c,c)$ and $\End^{*,\gr}(c)=\Hom^{*,\gr}(c,c)$.
	\end{itemize}
\end{enumerate}
\end{no}

Under the natural map $\oblv_{\scrC} :\scrC \to  \scrC^{\oblv} := \scrC \otimes_{\Vect^{\gr,c}} \Vect^c$, we have canonical isomorphism 
$$\oblv(\RHom^{\gr}(c,d)) \simeq \RHom(\oblv_{\scrC}(c),\oblv_{\scrC}(d)),$$ 
where $\textup{oblv}: \Vect^\gr \to \Vect$ is the natural map by forgeting the grading. For a $\Vect^{\gr,c}$-linear functor $F:\scrC \to \scrD$, denote by $F^{\oblv}:\scrC^\oblv \to \scrD^\oblv$ the induced functor.

If $\scrC= \Sh_{\gr,c,S}(X)$, then by assumption $\RHom^{\gr}(c,d)$ is pure for $c,d \in \scrC^{\hsw}$, namely 
$H^i(\Hom^{\gr}(c,d))$ is concentrated in graded degree $i$. Therefore $\REnd^{\gr}(c)$ is a formal graded dg algebra, for $c \in \scrC^{\hsw}$, namely we have an isomorphism of graded dg algebras $\REnd^{\gr}(c) \simeq \End^{*,\gr}(c)$. 

\begin{eg} Let $V$ be a finite dimensional vector space. Denote $A=\Sym(V[-2]\langle -2 \rangle )$ the graded symmetric algebra, with $V$ in cohmological degree $2$, and graded degree $2$.
	Then the category $\mathscr{C}=A$-$\textup{perf}^\gr$ of perfect graded dg-modules is a graded transversal category, with $\mathscr{C}^{w=k} \cap \mathscr{C}^{\heartsuit_{t}}=<A \langle k \rangle>$. 
	Put $A^!= \Sym(V^*[1]\langle 2 \rangle )$, then under Koszul duality: $A$-$\textup{perf}^\gr \simeq A^!$-$\textup{mod}_{\textup{fd}}^{\gr}=:\mathscr{D}$. 
	The transversal structure on $\mathscr{D}$ can be identified as $\mathscr{D}^{w=k} \cap \mathscr{D}^{\heartsuit_t}=<\CC\langle k \rangle>$.
	\end{eg}

\begin{defn}
Let $F:\mathscr{C} \to \mathscr{D}$ be a $\Vect^{\gr,c}$-linear functor between graded transversal $\infty$-categories. We say $F$ is \textit{trans-exact} if $F$ is both weight exact and $t$-exact.
\end{defn}

For any $\Vect^{\gr,c}$-linear category $\mathscr{C}$, and $c_1,c_2 \in \mathscr{C}$, denote by $\RHom^{\gr}(c_1,c_2) \in \Vect^\gr$ the enriched graded hom complex.

Denote the $\TCat_{\gr}$ be the $\infty$-category of graded transversal categories, with trans-exact functors, we have:

\begin{prop}
	\label{prop:transversal_to_additive} 
	The $\infty$-category $\TCat_{\gr}$ is equivalent to a $2$-category, moreover
natural functor $\TCat_{\gr} \to \Cat_1$, via $\mathscr{C} \mapsto \mathscr{C}^{\heartsuit_{w,t}}$, is locally faithful.
\end{prop}

\begin{proof}
By \cite[Theorem 2.2.9]{elmantoNilpotentExtensionsInfty2021}, there is an equivalence of $\infty$-categories: 
\beq 
            (-)^{\hsw}: \WCat \to \Cat_1^{\textup{add}} :  (-)^{\textup{fin}}
\eeq
where $\WCat$ denotes the $\infty$-category with a bounded weight structure, whose weight heart is isomorphic to an 1-category, and $ \Cat_1^{\textup{add}}$ denote the $\infty$-category of additive $1$-categories. This naturally upgrades to an equivalence:
\beq 
(-)^{\hsw}: \WCat_\gr \to \Cat_{1,\gr}^{\textup{add}} :  (-)^{\textup{fin}}
\eeq

where 	$\WCat_\gr$ denotes the $\infty$-category of $\Vect^{\gr,c}$-linear weight categories as above, and $\Cat_{1,\gr}^{\textup{add}}$ denote the $\infty$-category of $\Vect^{\gr,c,\hsw}$-linear additive $1$-categories.
Note that $\Cat_1^{\textup{add}},\Cat_{1,\gr}^{\textup{add}} $ are equivalent to a $2$-category. Therefore for any $\scrC, \scrD \in \TCat_{\gr}$, applying $(-)^{\hsw}$, we identify $\Hom_{\TCat_\gr}(\scrC, \scrD)$ with $\{F \in \Hom_{\Cat_{1,\gr}}(\scrC^\hsw, \scrD^\hsw): (F)^\textup{fin} \textup{ is t-exact.}\}$, which is an $1$-category. Therefore $\TCat_\gr$ is equivalent to a $2$-category. 
	
Let $\alpha:F \to G:\mathscr{C} \to \mathscr{D} $ be a 2-cell in  $\TCat_\gr $. Then $\alpha$ is uniquely determined by $\alpha^{\heartsuit^w} :F^{\heartsuit^w} \to G^{\heartsuit^w}:\mathscr{C}^{\heartsuit^w} \to \mathscr{D}^{\heartsuit^w}.$ We need to show it is further determinded by $\alpha^{\heartsuit^{w,t}} :F^{\heartsuit^{w,t}} \to G^{\heartsuit^{w,t}}:\mathscr{C}^{\heartsuit^{w,t}} \to \mathscr{D}^{\heartsuit^{w,t}}.$ Indeed,  any object $c \in \mathscr{C}^{\heartsuit_w}$ is isomoprhic to an object of the form $ \oplus c_i(i)$, for $c_i \in \mathscr{C}^{\heartsuit_{w,t}}$ . Then $\alpha^{\heartsuit_w}_c = \oplus \alpha^{\heartsuit_{w,t}}_{c_i}(i)$ as elements in $\Hom_{\calC^{\hsw}}(c,c)$ (this is required by definition for $2$-cells in $\Cat_{1,\gr}$). Therefore $\alpha^{\heartsuit_w}$ is determined by $\alpha^{\heartsuit_{w,t}}$.

\end{proof}

\section{Canonical induction/restriction functors} 
\label{sec:canonical_functors}

The goal of this section is to give a combinatorial description of the category character sheaves on reductive Lie algebras, together with the induction/restriction functors between them. The main results are Proposition~\ref{prop:lie_algebra_spectral} and its variation Proposition~\ref{prop:identify_invariant_category}. At first glance, these statements follow directly from \cite{lusztigCuspidalLocalSystems1988,riderFormalityLusztigGeneralized2021}. However, the main point here is to explain all the functors involved are canonically independent of choices of Borel/parabolic subgroups. For example, we explain the parabolic induction/restriction functors between character sheaves on Lie algebra is canonically independent of choice of the parabolic subgroups. Such statement is known for perverse character sheaves (for e.g. \cite{ginzburgInductionRestrictionCharacter1993}). The argument use crucially the fact that perverse sheaves form a 1-category, which make it hard to adapt to case of dg-categories. Our approach is to use mixed geometry, which provides a canonical graded lift of dg-category of character sheaves. The pure objects therein form an $1$-category, which we leverage to show the desired independence of parabolic. In more recent work of \cite{laumonDerivedLusztigCorrespondence2023}, the derived Springer correspondence is reinterpreted using weights, we expect that some of the statements in this section can be obtained from their results via Fourier transform.

The content of this section is organized as follows:
\begin{itemize}
	\item Section~\ref{sec:canonical_induction}: show the parabolic induction functor is canonically defined for the principal block.
	\item Section~\ref{sec:spectral_springer}: combinatorial description of the principal block
	\item Section~\ref{sec:canonical_general} \& \ref{sec:spectral_general}: repeat Section~\ref{sec:canonical_induction} \& \ref{sec:spectral_springer} for general cuspidal block.
	\item Section~\ref{localizing}: reinterpret the combinatorics using facet geometry.
\end{itemize}

\subsection{Canonical induction/restriction for principal block} 
\label{sec:canonical_induction} 
 Let $S$ be the set of irreducible character sheaves on $\frg/G$. 
 Put $\Ch_c(\frg/G)= \Sh_{c,S}(\frg/G)$ as in Appendix~\ref{secA:graded_lift}. 
 Then by \cite[Prop 3.5]{riderFormalityLusztigGeneralized2021}, 
 $S$ satisfies the Hom-purity assumption~(\ref{eqn:purity}). 
Denote $\Ch_{\gr,c}(\frg/G):= \Sh_{\gr,c,S}(\frg/G)$. Then $\Ch_{\gr,c}(\frg/G)^{\heartsuit_{w,t}} = \Sh_{c,S}(\frg/G)^{\heartsuit_t}$ \footnote{This has used the fact that $\Sh_{c,S}(\frg/G)^{\heartsuit_t}$ is actually semi-simple. See Example~\ref{eg:graded_sheaves}.} the category of perverse character sheaves. We have an equivalence 
 $\Ch_{\gr,c}(\frg/G)  \otimes_{\Vect^{\gr,c}} \Vect^{c} \simeq \Ch_c(\frg/G).$ 

Let $T \subset G$ be a maximal torus, and $B$ be a Borel subgroup of $G$, containing $T$. Denote by $\frt,\frb,\frg$ the corresponding Lie algebras, and by $\frt^\reg, \frb^\reg, \frg^\reg$ the regular semisimple elements. We have a commutative diagram:
$$
\xymatrix{ \frt^{\reg}/T  \ar[d]^{j_T}   \ar[rr]^{f}  &    &  \frg^{\reg}/G \ar[d]^{j_G}     \\
			\frt/T     &     \frb/B	 \ar[l]_{q} \ar[r]^{p}	&      \frg/G
}
$$
We denote the functors:
$$\Ind_B=\Ind_{T \subset B}^G := p_!q^* : \Ch_{\gr,c}(\frt/T) \to \Ch_{\gr,c}(\frg/G)$$
$$\Ind_B^{\hst}=\Ind_{T \subset B}^{G,\hst} := p_!q^* : \Ch_{c}(\frt/T)^\hst \to \Ch_{c}(\frg/G)^\hst$$
There is a canonical isomorphism of functors:
$$ \alpha^0_B : \Ind_B^{\heartsuit_t} = p_!q^* \simeq  j_{G!*}   f_*j_{T}^* : \Ch_c({\frt/T})^{\heartsuit_t} \to \Ch_c({\frg/G})^{\heartsuit_t} $$
Now let $B_1,B_2$ be two Borel subgroups containing $T$. Then we have a canonical isomorphism of functor:
$$\alpha^0_{B_1,B_2}:=  (\alpha^0_{B_2})^{-1} \circ  \alpha^0_{B_1}: \Ind_{ B_1}^{\heartsuit_t} \to \Ind_{ B_2}^{\heartsuit_t}$$

\begin{prop}
	The 2-cell $\alpha^0_{B_1,B_2}$ in $\Cat_1$  is in the image of the natural map:
	\beq
	\label{eq:restriction_alpha}
	(-)^{\hstw}: \Hom_{\TCat_\gr}(\Ind_{B_1}, \Ind_{B_2}) \to \Hom_{\Cat_1}(\Ind_{B_1}^{\hst}, \Ind_{B_2}^{ \heartsuit_{t}})  
	\eeq
\end{prop}	

\begin{proof} The data of a natural transformation between $\alpha:\Ind_{B_1} \to \Ind_{B_2} $ consist of the following:
	\begin{itemize}
		\item For any $F \in \Ch_{\gr,c}(\frt/T)$, a morphism $\alpha_F \in  \Hom^0_{\Ch_{\gr,c}(\frg/G)}(\Ind_{B_1}(F), \Ind_{B_2}(F))$.
		\item Such that for any $b:F_1 \to F_2$ in $\Ch_{\gr,c}(\frt/T)$ the diagram commutes:
\beq \xymatrix{ \Ind_{B_1}(F_1)  \ar[r]^{\alpha_{F_1}} \ar[d]^{\Ind_{B_1}b} & \Ind_{B_2}(F_2) \ar[d]^{\Ind_{B_2}b}   \\
	\Ind_{B_1}(F_2) \ar[r]^{\alpha_{F_2}}	     &    \Ind_{ B_2}(F_2) }
\eeq
\end{itemize}
Now since $\CC_{\frt/T}$ generate $\Ch_{\gr,c}(\frt/T)$ under direct sum, direct summand and $\Vect^{\gr,c}$-action. Therefore it suffices to construct $\alpha_{F}$ for $F=\CC_{\frt/T}$, such that  
for any morphism $b: \CC_{\frt/T} \to \CC_{\frt/T}\{i\} $, the diagram in $\Ch_{\gr,c}(\frg/G)$ commutes:
\beq
\label{eq:alpha_commute} \xymatrix{ \Ind_{B_1}(\CC_{\frt/T}) \ar[r]^{\alpha_{\CC_{\frt/T}}} \ar[d]^{\Ind_{B_1}b} & \Ind_{B_2} (\CC_{\frt/T}) \ar[d]^{\Ind_{B_2}b}   \\
		\Ind_{B_1}(\CC_{\frt/T}\{i\}) \ar[r]^{\alpha_{\CC_{\frt/T}}\{i\}}	     &    \Ind_{B_2}(\CC_{\frt/T}\{i\}) }
	\eeq
 Now the natural functor $\Ch_c(\frg/G)^{\hst} \simeq \Ch_{\gr}^{\hstw}(\frg/G) \to \Ch_{\gr,c}(\frg/G)$ induces a bijection on $\Hom$-set (both of them are isomorphic to $\CC[W]$):
  \beq  i:\Hom^0_{\Ch_c(\frg/G)^{\hst}}(\Ind_{B_1}^{\hst}\CC_{\frt/T}, \Ind_{B_2}^{\hst}\CC_{\frt/T})    \to     \Hom^0_{\Ch_{\gr,c}(\frg/G)}(\Ind_{B_1}\CC_{\frt/T}, \Ind_{B_2}\CC_{\frt/T})  \eeq	
Define $\alpha_{\CC_{\frt/T}}= i(\alpha^0_{B_1,B_2})$, we are left to show the commutativity of the diagram \eqref{eq:alpha_commute}. Applying the functor $- \otimes_{\Vect^{\gr,c}} \Vect^c$ it suffices to show the commutativity of the diagram:
\beq
\label{eq:alpha_commute_oblv} \xymatrix{ \Ind^{\oblv}_{B_1}(\CC_{\frt/T}) \ar[r]^{(\alpha_{\CC_{\frt/T}})^\oblv} \ar[d]^{(\Ind_{B_1}b)^\oblv} & \Ind^{\oblv}_{B_2} (\CC_{\frt/T}) \ar[d]^{(\Ind_{B_2}b)^\oblv}   \\
	\Ind^{\oblv}_{B_1}(\CC_{\frt/T}[i]) \ar[r]^{(\alpha_{\CC_{\frt/T}})^\oblv[i]}	     &    \Ind^{\oblv}_{B_2}(\CC_{\frt/T}[i]) }
\eeq
Here we have used the fact that $\oblv (\CC_{\frt/T}\{i\})  \simeq \CC_{\frt/T}[i]$, and $ \Ind^{\oblv}_{B_j}$ can be identified as an element of the form $\Ind^{\oblv}_{B_j}(b')$ for some $b' \in \Hom^0_{\Ch_{c}(\frt/T)}(\CC_{\frt/T}, \CC_{\frt/T}[i])$ ($b'$ is independent of $j=1,2$). 

Now consider the commutative diagram :
\beq 
\xymatrix{ \CC[W] \otimes H^*(\frt/T)  \ar[d]^{\simeq} \ar[rd] \\ \Hom^0_{\frg/G}(\Ind^{\oblv}_{B_1}(\CC_{\frt/T}),\Ind^{\oblv}_{B_2}(\CC_{\frt/T})) \otimes \Hom^*_{\frt/T}(\CC_{\frt/T}, \CC_{\frt/T}) \ar[r]^-{\theta} \ar[d] &  \Hom^*_{\frg/G}(\Ind^{\oblv}_{B_1}(\CC_{\frt/T}), \Ind^{\oblv}_{B_2}(\CC_{\frt/T}))  \ar[d]  \\
	\Hom^0_{\frg/G}(\Ind^{\oblv}_{B_1}(\CC_{\frt^{\reg}/T}),\Ind^{\oblv}_{B_2}(\CC_{\frt^{\reg}/T})) \otimes \Hom^*_{\frt^{\reg}/T}(\CC_{\frt^{\reg}/T}, \CC_{\frt^{\reg}/T}) \ar[r]^-{\theta^\reg} \ar[d]^{\simeq}  &  \Hom^*_{\frg^{\reg}/G}(\Ind^{\oblv}_{B_1}(\CC_{\frt^{\reg}/T}), \Ind^{\oblv}_{B_2}(\CC_{\frt^{\reg}/T}))     \\
	\CC[W] \otimes H^*(\frt^{\reg}/T) \ar[ur]  
}
\eeq
where the vertical arrows are induced by the pullback $\frt^{\reg}/T \to \frt/T$.
The map $\theta^{\reg}$ is an isomorphism since both $\Ind^{\oblv}_{B_1}$ and $\Ind^{\oblv}_{B_1}$ can be identified as the push forward along the $W$-cover $\frt^{\reg}/T \to \frg^{\reg}/G$.
Furthermore, under the identification: 
\beq \CC[W] \otimes H^*(\frt^{\reg}/T) \simeq \CC[W] \# H^*(\frt^{\reg}/T) \simeq 
\Hom^*_{\frg^{\reg}/G}(\Ind^{\oblv}_{B_1}(\CC_{\frt^{\reg}/T}), \Ind^{\oblv}_{B_2}(\CC_{\frt^{\reg}/T})) 
\eeq

The natural left and right action of  $H^*(\frt^{\reg}/T)$ on  $\Hom^*(\Ind^{\oblv}_{B_1}(\CC_{\frt^{\reg}/T}), \Ind^{\oblv}_{B_2}(\CC_{\frt^{\reg}/T}))$ is identified with the natural left and right $H^*(\frt^{\reg}/T)$-action on $\CC[W] \# H^*(\frt^{\reg}/T)$. 

Moreover, $\theta$ is also an isomorphism (see  \eqref{eq:induct_iso_prin} below), this implies similarly under the identification 
\beq \CC[W] \# H^*(\frt/T) \simeq \CC[W] \otimes H^*(\frt/T)   \simeq 
\Hom^*_{\frg/G}(\Ind^{\oblv}_{B_1}(\CC_{\frt/T}), \Ind^{\oblv}_{B_2}(\CC_{\frt/T})) 
\eeq
The left and right action of  $H^*(\frt/T)$ on  $\Hom^*(\Ind^{\oblv}_{B_1}(\CC_{\frt/T}), \Ind^{\oblv}_{B_2}(\CC_{\frt/T}))$ is identified with the left and right $H^*(\frt/T)$-action on $\CC[W] \# H^*(\frt/T)$. 
Now $ (\alpha_{\CC_{\frt/T}})^\oblv$ considered as an element in $\Hom^0_{\frg/G}(\Ind^{\oblv}_{B_1}(\CC_{\frt/T})$ is identified with $1 \otimes 1 \in  \CC[W] \# H^*(\frt/T)  $ (since by construction $(\alpha_{\CC_{\frt/T}})^\oblv$ is identified with the identity map between $j_{G!*} f_* \CC_{\frt^\reg/T} $ ) . Therefore
the square \eqref{eq:alpha_commute_oblv} commutes. By definition, the under $(-)^{\hstw}$, the image of $\alpha$ defined above is $\alpha^0_{B_1,B_2}$.
\end{proof}

Let $ \alpha_{B_1,B_2}$ be the preimage of $\alpha_{B_1,B_2}^{0}$ in (\ref{eq:restriction_alpha}), which is necessarily unique by Proposition~\ref{prop:transversal_to_additive}. Uniqueness also imply that  $\alpha_{B,B}= Id_{\Ind_B},$ and $ \alpha_{B_2,B_3} \circ \alpha_{B_1,B_2} = \alpha_{B_1,B_3}.$ 
Therefore we can view $\Ind_B,\alpha_{B_1,B_2}$ as defining canonically a functor 
\beq 
\label{eq:canonical_induction}
\Ind_T^G: \Ch_{\gr,c}(\frt/T) \to \Ch_{\gr,c}(\frg/G),
\eeq
 which we refer as the \textit{canonical induction} functor (for the principal block).

\subsection{Spectral identification via Springer theory}
\label{sec:spectral_springer}
In this section, we identify the canonical induction functor using explicit Lie theoretic data. 

Let $\Ch_{\gr,c}(\frg/G)_{\prin} \subset \Ch_{\gr,c}(\frg/G)$ be the full subcategory generated by Springer sheaves under $\Vect^\gr$-action. Fix a maximal torus $T$, for a choice of Borel $B$ containing $T$. By \cite[Prop 3.2]{lusztigCuspidalLocalSystems1988}, we have an isomorphism: 
\beq \label{eq:induct_iso_prin}
 \End^{0,\gr}(\Ind^{\gr}_B(\CC_{\frt/T})) \otimes  \End^{*,\gr}({\CC_{\frt/T}}) \to \End^{*,\gr}(\Ind^{\gr}_B(\CC_{\frt/T})) \eeq

This gives isomorphisms of graded dg algebras:
$$   \C[W] \# \Sym (\frt^*[-2] \langle -2 \rangle ) \simeq \End^{*,\gr}(\Ind^{\gr}_B(\CC_{\frt/T}))  \simeq \REnd^\gr(\Ind^{\gr}_B({\CC_{\frt/T}})) $$
where the last isomorphism is by formality. Therefore we have an equivalence of graded transversal categories:
$$  \LL_{\frb \subset \frg,\prin}:   \C[W] \# \Sym (\frt^*[-2]\langle -2 \rangle)  \perf^\gr \simeq \Ch_{\gr,c}(\frg/G)_{\prin} $$

Given two Borel subgroup $B_1, B_2$, mimic the construction in Section~\ref{sec:canonical_induction}, we can define a $2$-cell $\beta_{B_1,B_2}: \LL_{\frb_1 \subset \frg,\prin} \to \LL_{\frb_2 \subset \frg,\prin}$. Therefore we have a canonical functor 
\beq
  \LL_{\frg,\prin}:  \C[W] \# \Sym (\frt^*[-2]) \perf^\gr \simeq \Ch_{\gr,c}(\frg/G)_{\prin} .
  \eeq
   Moreover, the diagram naturally commutes:
\beq  
\xymatrix{ \Sym (\frt^*[-2]\langle -2 \rangle) \perf^\gr \ar[d]^{\Ind} \ar[r]^-{\LL_{\frt}}  & \Ch_{\gr,c}(\frt/T) \ar[d]^{\Ind_T^G} \\
\C[W] \# \Sym (\frt^*[-2]\langle -2 \rangle) \perf^\gr \ar[r]^-{\LL_{
		\frg,\prin}} &  \Ch_{\gr,c}(\frg/G)_{\prin}
}
\eeq
Let us spell out this commutativity more explicitly: this means for any choice of Borel $B$ and $R$, we have a $2$-cell $\gamma_{B,R}$:
\beq  
\xymatrix{ \Sym (\frt^*[-2]\langle -2 \rangle) \perf^\gr \ar[d]^{\Ind} \ar[r]^-{\LL_\frt}  & \Ch_{\gr,c}(\frt/T) \ar[d]^{\Ind_{T \subset R}^G} \\
	\C[W] \# \Sym (\frt^*[-2]\langle -2 \rangle) \perf^\gr \ar[r]^-{\LL_{\frb \subset \frg,\prin}} &  \Ch_{\gr,c}(\frg/G)_{\prin}
}
\eeq
and it is compatible with composition with $\beta_{B,B'}$ on the bottom, and with $\alpha_{R,R'}$ on the right.

\subsection{Canonical induction/restriction for general blocks}
\label{sec:canonical_general} In this and the next section, we describe similar results for generalized Springer theory, see \cite{lusztigCuspidalLocalSystems1988} for reference. The proofs are similar and we shall omit them.
 Firstly, we fix $\kappa=(M, O, \mathcal{L})$ a cuspidal data of $G$, i.e, $M$ is a Levi subgroup of $G$, $O$ a nilpotent orbit of $M$, and $\calL$ a cuspidal local system on $O$. Put $\frz_\kappa = \frz_{\frakm}$ the center of $\frakm$, and $W_G^\kappa= N_G(M)/M$ the relative Weyl group. 
Denote $\Ch_{\gr,c}(\frakm/M)_\kappa \subset \Ch_{\gr,c}(\frakm/M) $ the subcategory generated by $\overline{\calL}:= \IC( \frakm/M ,\CC \frz_\kappa \boxtimes \calL) $, and denote by $\Ch_{\gr,c}(\frg/G)_\kappa \subset \Ch_{\gr,c}(\frg/G)$ the full subcategory generated by $\Ind_{L \subset Q}^G(\overline{\calL})$ (where $Q$ is some parabolic subgroup of $G$, with Levi factor $L$). 
Now put $\frz_\kappa^{\reg}= \{ x \in \frz_\frakm: C_G(x)=M\}$, 
put $\frakm^\reg_\kappa= \frz_\frakm^{\reg} \times O$,
and $\frg^{\reg}_{\kappa}/G$ the image of $\frakm^{\reg}_{\kappa}/M$ under the map $\frakm/M \to \frg/G$. We have commutative diagram:
$$
\xymatrix{ \frakm^{\reg}_\kappa/M  \ar[d]^{j_L}   \ar[rr]^{f}  &    &  \frg^{\reg}_\kappa/G \ar[d]^{j_G}     \\
	\frakm/M     &     \frq/Q	 \ar[l]_{q} \ar[r]^{p}	&      \frg/G
}
$$
There is a canonical isomorphism of functors:
\beq \alpha^0_{Q,\kappa} : \Ind_Q^{\heartsuit_t} := \Ind_{M \subset Q}^{G,\heartsuit_t} := p_!q^* \simeq  j_{G!*}   f_*j_{M}^* : \Ch_c({\frakm/M})^{\heartsuit_t}_\kappa \to \Ch_c({\frg/G})^{\heartsuit_t}_\kappa 
\eeq
\beq 
\alpha'^{0}_{Q,\kappa} : \Res_Q^{\heartsuit_t} := \Res_{M \subset Q}^{G,\heartsuit_t} := q_*p^! \simeq  j_{M!*}   f^*j_{G}^* : \Ch_c({\frg/G})^{\heartsuit_t}_\kappa \to \Ch_c({\frakm/M})^{\heartsuit_t}_\kappa 
\eeq

Now let $Q_1,Q_2$ be two Borel subgroups containing $M$. Then we have a canonical isomorphism of functor:
 \beq  \alpha^0_{Q_1,Q_2,\kappa}:=  (\alpha^0_{Q_2,\kappa})^{-1} \circ  \alpha^0_{Q_1,\kappa}: \Ind_{M \subset Q_1}^{G,\heartsuit_t} \to \Ind_{M \subset Q_2}^{G,\heartsuit_t}
 \eeq
 
  \beq  \alpha'^0_{Q_1,Q_2,\kappa}:=  (\alpha'^0_{Q_2,\kappa})^{-1} \circ  \alpha'^0_{Q_1,\kappa}: \Res_{M \subset Q_1}^{G,\heartsuit_t} \to \Res_{M \subset Q_2}^{G,\heartsuit_t}
 \eeq

\begin{prop}
	\label{prop:alpha_image_general}
\begin{enumerate}
	\item 	The 2-cell $\alpha^0_{Q_1,Q_2,\kappa}$ in $\Cat_1$  is in the image of the natural map:
	\beq
	\label{eq:restriction_beta}
	\Hom_{\TCat}(\Ind_{M,Q_1,\gr}^{G}, \Ind_{M,Q_2,\gr}^{G}) \to \Hom_{\Cat_1}(\Ind_{M,Q_1}^{G, \hst}, \Ind_{M,Q_2}^{G, \heartsuit_{t}})  
	\eeq
	\item 	The 2-cell $\alpha'^0_{Q_1,Q_2,\kappa}$ in $\Cat_1$  is in the image of the natural map:
	\beq	\Hom_{\TCat}(\Res_{M,Q_1,\gr}^{G}, \Res_{M,Q_2,\gr}^{G}) \to \Hom_{\Cat_1}(\Res_{M,Q_1}^{G, \hst}, \Res_{M,Q_2}^{G, \heartsuit_{t}})  
	\eeq
\end{enumerate}

\end{prop}	
Similarly to (\ref{eq:canonical_induction}), we denote the canonical functors:
\beq
\Ind_{M,\kappa}^{G}: \Ch_{\gr,c}(\frakm/M)_{\kappa} \rightleftarrows \Ch_{\gr,c}(\frg/G)_{\kappa}: \Res_{M,\kappa}^G .
\eeq

Now let $L$ be a Levi subgroup of $G$, and denote $\frK_L$ the set of cuspidal data of $L$ (up to conjugacy). We have a natural inclusion $\frK_L \subset \frK_G$.  Let $\kappa =(M,O,\calL) \in \frK_L$ be a cuspidal data of $L$, we can also view $\kappa$ as a cuspidal data for $G$. A similar argument will give a canonically defined functor:
\beq
\Ind_{L,\kappa}^{G}: \Ch_{\gr,c}(\frl/L)_{\kappa} \rightleftarrows \Ch_{\gr,c}(\frg/G)_{\kappa}: \Res_{L,\kappa}^G.
\eeq
Taking direct sum over all cuspidal data, we have functors:
\beq
 \Ind_L^G: \Ch_{\gr,c}(\frl/L)  \simeq \oplus_{\kappa \in \frK_L} \Ch_{\gr,c}(\frl/L)_{\kappa} \rightleftarrows \Ch_{\gr,c}(\frg/G) \simeq \oplus_{\kappa \in \frK_G} \Ch_{\gr,c}(\frg/G)_{\kappa} : \Res_{L}^G,
\eeq
where $\Res^G_L$ sends all summands $\kappa \in \frK_G \backslash \frK_L$ to $0$.
\begin{rmk}
	\label{rmk:canonical}
	Similar argument as in Proposition~\ref{prop:alpha_image_general} above shows that for any parabolic $P$ of $L$, there are canonical isomorphisms of functors $\Ind_L^G \simeq \Ind_{L \subset P}^G$ and $\Res_L^G \simeq \Res_{L \subset P}^G$. In this sense, we see that the functors $\Ind_L^G$ and $\Res_L^G$ are independent of the choice of parabolic subgroup, and therefore we refer $\Ind_L^G, \Res_L^G$ as the \textit{canonical induction and restriction functor}. 
\end{rmk}

\subsection{Spectral identification via generalized Springer theory}
\label{sec:spectral_general}
We continue to give an analogous spectral description as in Section~\ref{sec:spectral_springer}.  Let $\kappa=(M,O,\calL)$ be a cuspidal datum of $G$, and $Q$ be a parabolic subgroup with Levi factor $L$. We have a natural map:
\beq
\label{eq:induction_isomorphism}
  \End^{0,*}(\Ind^{\gr}_Q(\overline{\calL})) \otimes  \End^{*,\gr}({\overline{\calL}}) \to \End^{*,\gr}(\Ind^{\gr}_Q(\overline{\calL})) 
  \eeq 

By \cite[Prop 3.2]{lusztigCuspidalLocalSystems1988}, \eqref{eq:induction_isomorphism} is an isomorphism and can be identified as maps of graded vector spaces:
\beq
  \C[W_G^\kappa] \otimes  \Sym (\frz_\kappa^*[-2] \langle -2 \rangle )  \simeq \End^{*,\gr}(\Ind^{\gr}_Q({\overline{\calL}})) 
\eeq
Moreover, the algebra structure on $\REnd^\gr(\Ind^{\gr}_Q({\overline{\calL}})$ can be identified as the smash product on the LHS, therefore, we have an isomorphism of graded dg algebras:

\beq
  \C[W_G^\kappa] \# \Sym (\frz_\kappa^*[-2]  \langle -2 \rangle)  \simeq \End^{*,\gr}(\Ind^{\gr}_Q({\overline{\calL}}))  \simeq \REnd^\gr(\Ind^{\gr}_Q({\overline{\calL}}))
\eeq
which further gives an equivalence of transversal categories:
\beq
 \LL_{\frq \subset \frg,\kappa}:   \C[W_G^\kappa] \# \Sym (\frz_\kappa^*[-2]  \langle -2 \rangle)  \perf^\gr \simeq \Ch_{\gr,c}(\frg/G)_{\kappa} 
\eeq
Where $\modu^\gr_{\perf}$ denotes the category of graded perfect modules.
Now for two parabolic subgroups $Q_1, Q_2$ containing $M$, there is a $2$-cell $\beta_{Q_1,Q_2}: \LL_{Q_1 \subset G,\kappa} \to \LL_{Q_2 \subset G,\kappa}$ defined similarly. Therefore we have a canonical functor:
\beq 
\label{eq:kappa_spectral}
 \LL_{\frg,\kappa}:  \C[W_G^\kappa] \# \Sym (\frz_\kappa^*[-2]) \perf^\gr \simeq \Ch_{\gr,c}(\frg/G)_{\kappa} 
\eeq

Moreover, the diagram naturally commutes:
\beq  
\xymatrix{ \Sym (\frz_\kappa^*[-2] \langle -2 \rangle) \perf^\gr \ar@<-.5ex>[d]_{\Ind} \ar[r]^-{\LL_{M,\kappa}}  & \Ch_{\gr,c}(\frakm/M)_\kappa \ar@<-.5ex>[d]_{\Ind_{M,\kappa}^G} \\
	\C[W_G^\kappa] \# \Sym (\frz_\kappa^*[-2] \langle -2 \rangle) \perf^\gr \ar@<-.5ex>[u]_{\Res} \ar[r]^-{\LL_{G,\kappa}} &  \Ch_{\gr,c}(\frg/G)_{\kappa} \ar@<-.5ex>[u]_{\Res_{M,\kappa}^G}
}
\eeq

Taking direct sum over all cuspidal data $\kappa \in \frK_G$ of \eqref{eq:kappa_spectral}, we have

\begin{prop}
	\label{prop:lie_algebra_spectral}
There is a canonical equivalence:
\beq
\label{eq:spectral}
\LL_{\frg}:  \oplus_{\kappa \in \frK} \C[W_G^\kappa] \# \Sym (\frz_{\kappa}^*[-2] \langle -2 \rangle) \perf^\gr \simeq \Ch_{\gr,c}(\frg/G) .
\eeq

Moreover, let $L$ be a Levi subgroup of $G$ containing $T$, the diagram naturally commutes:
\beq  
\xymatrix{ \oplus_{\kappa \in \frK_L} \C[W_L^\kappa] \#  \Sym (\frz_\kappa^*[-2] \langle -2 \rangle) \perf^\gr \ar@<-.5ex>[d]_{\Ind} \ar[r]^-{\LL_{\frl}}  & \Ch_{\gr,c}(\frl/L) \ar@<-0.5ex>[d]_{\Ind_{L}^G} \\
\oplus_{\kappa \in \frK_G}	\C[W_G^\kappa] \# \Sym (\frz_\kappa^*[-2] \langle -2 \rangle) \perf^\gr  \ar[r]^-{\LL_{\frg}} \ar@<-.5ex>[u]_{\Res} &  \Ch_{\gr,c}(\frg/G) \ar@<-.5ex>[u]_{\Res^G_L}
}
\eeq
where $\Ind$ and $\Res$ are the induction and restriction along $W_L^\kappa  \subset W_G^\kappa$.
\end{prop}

\subsection{Localizing $\Ch(\frg/G)$ over the apartment}
\label{localizing}
In this section, we describe categories of character sheaves on Lie algebras as global sections of certain combinatorial defined sheaves on the affine space $X_*(T) \otimes \CC$.

\subsubsection{Sheaves of categories on affine spaces}
\label{sec:categories_on_affine_spaces}
We collect some facts about sheaves of categories in Appendix~\ref{secA:sheaves_of_categories}.
Let $\mathbb{E}$ be a finite-dimensional complex affine space,  $\frF$ a locally finite set of affine subspaces of $\mathbb{E}$. 
For any $\epsilon \in \frF$, we put $\scrQCoh_\epsilon$ the constant sheaf of category on $\epsilon$ with value $\QCoh(\frz_\epsilon[1])$, for $\frz_\epsilon$ the tangent space of $\epsilon$. Let $f: \frG \to \frF$ a map between sets,
denote $\scrQCoh_{\frG}=\prod_{c \in \frG} \scrQCoh_{f(c)}.$ Let $K$ be a discrete group of affine linear transformations acting properly discontinuously on $\mathbb{E}$, 
and $\frF$ is stable under $W$-action. Assume also that $K$ acts on $\frG$, such that $f$ is $K$-equivariant, then 
$\scrQCoh_\frG$ is naturally a $K$-equivariant sheaf on $\EE$. Hence for any $K$-invariant open subset $U$, we have natural $K$-action on $\Gamma(U,\scrQCoh_\frG)$. Denote the invariant category by $\Gamma(U,\scrQCoh_\frG)^K$. We define $\scrCoh_\frG , \scrPerf_\frG, \scrCoh^\gr_\frG$ similarly, by replacing $\QCoh(\frz_\epsilon[1])$ in the above definition by $\Coh(\frz_\epsilon[1]), \Perf(\frz_\epsilon[1]),\Coh(\frz_\epsilon[1]\langle  2 \rangle )$ respectively.

\begin{rmk}
	\label{rmk:invariant}
	We can write $\Gamma(\EE,\scrQCoh_\frG)^W$ more concretely as follows: 
	pick $[\frG//K]$ a set of representatives of $\frG//K$, 
	and denote $K_c$ the stabilizer of $K$ at $c \in \frC$. 
	Then we have equivalence of categories:  $$\Gamma(\EE,\scrQCoh_\frG)^W 
	\simeq  (\prod_{c \in \frG} \Sym(\frz_{f(c)}[1]) \textup{-mod} )^W
	\simeq \prod_{c \in [\frG//K]} \CC[K_c]\#\Sym(\frz_{f(c)}[1]) \textup{-mod}.$$
\end{rmk}

Let  $K' \subset K$, and $\frG' \subset \frG$ a $K'$-stable subset. We have natural pair of adjoint functors:
$$      \Ind_{\frG'/K'}^{\frG/K} : \Gamma(\EE, \scrQCoh_{\frG'})^{K'} \rightleftarrows           \Gamma(\EE, \scrQCoh_{\frG})^K  :\Res^{\frG/K}_{\frG'/K'}         $$ 
When the context is clear, we shall simply denote them by $\Ind, \Res$.

\begin{no} 
	\label{no:lie}
	We collect some Lie-theoretic notations:
	\begin{enumerate}
	\item $\Phi$ the set of roots of $(G,T)$.
	\item  Each $\alpha \in \Phi$ defines a linear function on $\frt$. Let $H_\alpha= \{x \in \frt: \alpha(x)=0\}$, and $s_\alpha$ the reflection about $H_\alpha$ defined by $s_{\alpha}(x)=x-\alpha(x) \alpha^{\vee}$.
	\item 	For  $\alpha \in \Phi$, put $\frg_\alpha \subset \frg$ the root space.
	\item  For $I$ a facet of $\frt$ (of the hyperplane arrangement given by $H_\alpha, \alpha \in \Phi$), put $\epsilon_I =  \cap_{\{\alpha \in \Phi: \alpha(I)=0\}} H_\alpha$ the linear subspace of $\frt$.  
	\item Define the set $\frS = \{ \epsilon_I :  I  \textup{ a facet as in (4)} \}$ of linear subspaces of $\frt$.
	\item For $\epsilon \in \frS$, put $\Phi_\epsilon:=\{ \alpha \in \Phi: \alpha(\epsilon) =0 \}$,  $W_\epsilon= \langle s_\alpha: \alpha  \in \Phi_{\epsilon} \rangle \subset W$, and $W^{\epsilon}= N_W(W_\epsilon)/W_\epsilon$. 
	\item Put $G_\epsilon $ the connected reductive subgroup of $G$ containing $T$ with roots $\Phi_\epsilon$,  so that it has Weyl group $W_\epsilon$.
	\item  Put $C_G$ the set of isomorphism classes of cuspidal sheaves on $\calN_G/G$, where $\calN_{G}\subset \frg$ is the nilpotent cone of $G$. 
	\item 	 $\frC=\{(\epsilon,B,F): \epsilon \in \frS, F \in C_{G_\epsilon}, B \supset T$  a Borel subgroup of $G_\epsilon $\}.
	\item We also use the notation $W_G,\frS_G, \frC_G,$ etc, to emphasis the dependence on $G$.
	\item For $I$ a facet, 
	put $G_I=G_{\epsilon_I}, W_I=W_{\epsilon_I},  W^{I}=W^{\epsilon_I}, \frz_I= \frz_{\epsilon_I}=\frz_{\frg_I}$ the center of $\frl_I$, $C_I=C_{{G_{I}}}$, $\tilfrC_I=\tilfrC_{G_I}$, etc.
\end{enumerate}
\end{no}

Now in the notation of Section~\ref{sec:categories_on_affine_spaces}, take $\EE=\frt$, and $f:\frC_G \to \frS_G$, via $(\epsilon,B,F) \mapsto  \epsilon$. Note that $f$ is naturally $W$-equivariant, therefore $\scrQCoh_{{\frC}_G}, \Gamma(\frt, \scrQCoh_{{\frC}_G})^W$ etc, are defined.

We have the following reformulation of Proposition~\ref{prop:lie_algebra_spectral}:

\begin{prop}
	\label{prop:identify_invariant_category}
	There is a canonical equivalence:
	\beq
	\label{eq:spectral_gr}
	\LL_{\frg}:  \Gamma(\frt, \scrCoh^{\gr}_{\frC_G})^{W_G} \simeq \Ch_{\gr,c}(\frg/G) .
	\eeq 

Moreover, let $L \subset G$ be a Levi subgroup containing $T$, the diagram naturally commutes:
	\beq  
	\label{spectral_gr_commute}
	\xymatrix{ \Gamma(\frt, \scrCoh^\gr_{\frC_L})^{W_L} \ar@<-.5ex>[d]_{\Ind} \ar[r]^-{\LL_{\frl}}  & \Ch_{\gr,c}(\frl/L) \ar@<-0.5ex>[d]_{\Ind_{L}^G} \\
	\Gamma(\frt, \scrCoh^\gr_{\frC_G})^{W_G} \ar[r]^-{\LL_{\frg}} \ar@<-.5ex>[u]_{\Res} &  \Ch_{\gr,c}(\frg/G) \ar@<-.5ex>[u]_{\Res^G_L}
	}
	\eeq

\end{prop}

\begin{proof}
We claim there is bijection between $\frK_G$ and $\frC_G//W_G$. Indeed, fix a Weyl chamber $A$ corresponds to a Borel subgroup $B$ of $G$, denote $\frD_G:=\{(J,F): J$\textup{ a facet of }$A$, $F \in C_{J}$\}. Then by \cite[Proof of Theorem 1.2]{liDerivedCategoriesCharactera}, we have an bijection $\frK_G \simeq \frD_G$. Now we claim the natural map 
	$h:\frD_G \to \frC_G//W_G$, $(J,F)\mapsto(\epsilon_J,B \cap G_J,F)$ is bijective. Surjectivity: by definition $(\epsilon_J,B \cap G_J,F)$ is in the image of $h$, the group $W_J$ action fix $\epsilon_J$ and $F$, but acts transitively on all Borel subgroup $R$ of $G_J$, therefore all elements of the form $(\epsilon_J, R, F)$ are in the image of $h$, now $h$ is surjective by $W_G$-action. Injectivity: now suppose $h(J,F)=h(J',F')$, i.e  there is $w \in W_G$, such that $w(\epsilon_J, B \cap G_J,F)= (\epsilon_{J'}, B \cap G_{J'},F')$ in $\frC_G$. Denote by $P_J$ the parabolic subgroup of $G$ with Levi factor $G_J$. Then $Ad_w(P_J)$ and $P_{J'}$ are both parabolic subgroup of $G$ with Levi $G_{J'}$, which supports a cuspidal sheaf, therefore $Ad_w(P_J)$ and $P_{J'}$ are conjugate in $G$ by \cite[Lemma 3.1]{liDerivedCategoriesCharactera}. Hence $P_J=P_J'$ and $J = J'$. Now we must have $w \in N_{W_G}(W_J)$, therefore $F'=w(F)=F$, and the injectivity holds.
	Finally, for $c=(\epsilon,R,F) \in \frC_G$, the composition $W_c \hookrightarrow N_{W_G}(W_\epsilon) \to N_{W_G}(W_\epsilon)/W_\epsilon= W^{\epsilon}$ is an isomorphism. Therefore \eqref{eq:spectral_gr} holds by Remark~\ref{rmk:invariant} and Proposition~\ref{prop:lie_algebra_spectral} (together with Koszul duality). It is easy to check the bijection $\frK_G \simeq \frC_G//W_G$ is compatible with the natural map from $L$ to $G$ on both sides, this implies \eqref{spectral_gr_commute}.
\end{proof}

\section{Spectral description of character sheaves on reductive groups}
\label{sec:gluing}

In this chapter, we calculate $\Ch(G/G)$ by gluing (i.e taking limit) the identifications on Lie algebras as in Proposition~\ref{prop:identify_invariant_category}. The contents are organized as follows:
\begin{itemize}
	\item  Section~\ref{sec:diagram}: define the diagram over which we take limits.
	\item  Section~\ref{sec:auto_glue}: define the gluing functor on the automorphic side, and identify its limits as $\Ch(G/G)$.
	\item Section~\ref{sec:spec_glue}: define the gluing functor on the spectral side, and compute explicitly its limits.
	\item Section~\ref{sec:match}: match up the functors on both sides and deduce our main theorem.
\end{itemize}

\subsection{Organizing diagram}
\label{sec:diagram} Let $LG=G(z,z^{-1})$ and $L\frg=\frg(z,z^{-1})$ be the loop group and loop Lie algebra.
We first introduce analogous notation as in Notation~\ref{no:lie} for the affine case:

\begin{no} 
	\label{no:lie_affine}
	\begin{enumerate}
		\item $\widetilde{\Phi}= \ZZ \times {\Phi} $ the set of affine roots. $\widetilde{W} = X_*(T) \rtimes W$ the extended affine Weyl group.
		\item  Each $\alpha =(n,\ov{\alpha}) \in \widetilde{\Phi}$ defines an affine linear function $\alpha:=\ov{\alpha}+n$. Let $H_\alpha= \{x \in \frt: \alpha(x)=0\}$, and $s_\alpha$ the reflection about $H_\alpha$ defined by $s_{\alpha}(x)=x-\alpha(x)\ov\alpha^{\vee}$.
		\item 	For  $\alpha =(n,\ov{\alpha}) \in \widetilde{\Phi}$, put $\frg_\alpha= \frg_{\ov\alpha} z^n \subset L\frg$ the root space.
		\item  For $I$ an affine facet of $\frt$ (i.e a facet of the hyperplane arrangement given by affine hyperplanes $H_\alpha, \alpha \in \widetilde{\Phi}$), put $\epsilon_I =  \cap_{\{\alpha \in \Phi: \alpha(I)=0\}} H_\alpha$ the linear subspace of $\frt$.  
		\item Define the set $\tilfrS = \{ \epsilon_I :  I  \textup{ a facet as in (4)} \}$ of linear subspaces of $\frt$.
		\item For $\epsilon \in \frS$, put $\Phi_\epsilon:=\{ \alpha \in \Phi: \alpha(\epsilon) =0 \}$,  $W_\epsilon= \langle s_\alpha: \alpha  \in \Phi_{\epsilon} \rangle \subset W$, and $\tilW^{\epsilon}= N_{\tilW}(W_\epsilon)/W_\epsilon$. 
		\item Put $G_\epsilon $ the connected reductive subgroup of $LG$ containing $T$ with roots $\Phi_\epsilon$,  so that it has Weyl group $W_\epsilon$.
		\item 	 $\tilfrC=\{(\epsilon,B,F): \epsilon \in \tilfrS, F \in C_{G_\epsilon}, B \supset T$  a Borel subgroup of $G_\epsilon $\}. Note that there is natural bijection between $\tilfrC$ with $\tilfrC_G$ in the introduction. We shall by abuse of notations, denote them by the same symbol.
		\item We also use the notation $\tilW_G,\tilfrS_G, \tilfrC_G,$ etc, to emphasis the dependence on $G$.
		\item For $I$ an affine facet, 
		put $G_I=G_{\epsilon_I}, W_I=W_{\epsilon_I},  W^{I}=W^{\epsilon_I}, \frz_I= \frz_{\epsilon_I}=\frz_{\frg_I}$ the center of $\frl_I$, $C_I=C_{{G_{I}}}$, etc.
	\end{enumerate}
\end{no}

\begin{defn}

\begin{enumerate}
	\item Denote $\frF_G$ the $(2,1)$-category with 
	\begin{itemize}
		\item Objects: affine facets $I$ in $\frt_\RR$;
		\item Morphisms: $w: I \to J$, for $ w \in \tilW$, and $I \subset \overline{w(J)}$;
		\item 2-morphisms: $u : w \rightarrow wu : I \to J$, for $u \in W_I$
	\end{itemize}	
	\item Let $L \subset G$ be a Levi subgroup, that contains $T$. For $I \in \frF_G$, denote by $I_L \in \frF_L$ the facet containing $I$. Let $\frF_{L \subset G}$ be the category with 
	\begin{itemize}
		\item Objects: $I \in \frF_G$;
		\item Morphism: $w: I  \to J$, for all $w \in \tilW_L$,  and $I \subset \overline{w(J)}$.
		\item 2-morphism: $u: w \Rightarrow wu : I \to J$, for all $u \in W_{I_L}.$
		\end{itemize}	
\end{enumerate}

\end{defn}

We have natural functors
\begin{equation}
	\xymatrix@R=1pt{    \frF_G     &    \frF_{L \subset G} \ar[r]^{\alpha} \ar[l]_{\beta}  &    \frF_L         \\
				(I,w,u)      &      (I,w,u)   \ar@{|->}[l] \ar@{|->}[r]           &       (I_L,w,u)        }
\end{equation}

\begin{rmk}
	\label{rmk:F_G}
\begin{enumerate}
	\item For any $w_1, w_2$ two 1-morphisms in $\frF_G$ or $\frF_{L \subset G}$, the set of two morphisms between $w_1$ and $w_2$ is either $\emptyset$ or a singleton. Therefore both $\frF_G$ and $\frF_{L \subset G}$ are equivalent to 1-categories.
	\item When $G$ is simple and simply-connected with a fixed alcove $A$, denote $\frF_A$ the poset of (non-empty) facets of $A$, with $I \to J$ if $I \subset \overline{J}$. Then the natural map $\frF_A \to \frF_G$ is an equivalence.
\end{enumerate}
\end{rmk}

\begin{lem}
	\label{lem:fiber_alpha_contractible}
	The fibers of $\alpha$ have contractible geometric realizations.
\end{lem}
\begin{proof}
	$|\alpha^{-1}(J)| = | \{I \in \frF_G : I \subset   J\} | \simeq J$ is contractible.
\end{proof}

\subsection{Automorphic side}
\label{sec:auto_glue}
We defined the gluing functor on the automorphic side: the ungraded one in Section~\ref{auto_ungraded}, and the graded one in Section~\ref{auto_graded}.

\subsubsection{The functor $\Ch^{open}$}
\label{auto_ungraded}

The result in this section is a variation of \cite[Theorem 5.40]{liUniformizationSemistableBundles2021}, which identity $\Ch(G/G)$ as a limit of character sheaves on various (open subsets of) Lie algebras.
We refer to \cite[Section 5]{liUniformizationSemistableBundles2021} for definitions and notations.
For $I \in \frF_G$, denote by $G_I \subset LG$ the corresponding Levi subgroup. 
There is a twisted conjugation action of $G_I$ on its Lie algebra $\frg_I$ (by gauge action), we denote the quotient stack by $\frg_I/'G_I$. The statement in previous sections for the usual adjoint quotient stack $\frg/G$ still holds for the twisted quotient.
We have the twisted characteristic polynomial map $\chi'_I: \frg_I \to \frt//W_I$, put $St_I \subset \frt_\RR$ the star of $I$, $V_I= St_I \times i\frt_\RR \subset \frt$, and $U_{G,I}=\chi'{}^{-1}(V_I//W_I) \subset \frg_I$. 
Fix $\dot{w}$ a lift $w \in W$ into $N_G(T)$, this gives a lift of set $\widetilde{W} \to LG$ .   We have a functor $\sfU_G: \frF_G \to \Stack$ \footnote{As in \cite{liUniformizationSemistableBundles2021}, 
	one can choose the lift uniformly by picking a surjective group homomorphism $\dot{W} \to W$, and introduce 2-morphisms to cancel the extra 1-morphisms. 
	However, for simplicity, we shall not proceed in this direction.}, via 
\begin{itemize}
	\item On objects:   $I \mapsto U_{G,I}/'G_I$;
	\item On 1-morphisms:  $\{w: I \to w(I) \} \mapsto Ad'_{\dot{w}}  : U_{G,I}/'G_I \to U_{G,w(I)}/'G_{w(I)}; $
	\item On 2-morphisms:  $\{u: w_1 \to w_2  \}   \mapsto $ the inner automorphism by conjugation of ${\dot{w}_2 \dot{w}_1^{-1}}$.
\end{itemize}
The pullback functor $Ad'_{\dot{w}}{}^*$ preserves the category of character sheaves, and induces a functor: 
$$\mathsf{Ch}^{\open}_G :  \frF_G \to \Cat, \qquad I \mapsto \Ch( U_{G,I}/'G_I).$$
For $I \in \frF_G$, put $L_I= L_{I_L} $, $U_{L,I}= U_{L,I_L} $, then we have $ U_{G,I} \cap L_I \subset U_{L,I}$. Define similarly the functor:
	 $$\mathsf{Ch}^{\open}_{L \subset G} :  \frF_{L \subset G} \to \Cat, \qquad I \mapsto \Ch( U_{G,I} \cap \frl_I /'L_I).$$

Let $P \subset G$ be a parabolic subgroup with Levi subgroup $L$. For $I \in \frF_G$, put $P_I=G_I \cap P$.  We have the correspondence:   $ U_{G,I} \cap \frl_I /'L_I  \xleftarrow{q}    U_{G,I} \cap \frp_I /'P_I  \xrightarrow{p} U_{G,I}/'G_I   $.   Put $\Res_{P_I}=p_* q^! : \Ch( U_{G,I}/'G_I) \to \Ch( U_{G,I} \cap \frl_I /'L_I).$ Define natural transformations:
\begin{itemize}
	\item $\sfRes_P : \mathsf{Ch}^{\open}_G \circ \beta  \Rightarrow \mathsf{Ch}^{\open}_{L \subset G} $,  via  $ I \mapsto \Res_{P_I}:  \Ch(U_{G,I}/'G_I) \to \Ch(U_{G,I} \cap \frl_I/'L_I)$;
	\item $\sfRef: \mathsf{Ch}^{\open}_L \circ \alpha \Rightarrow \mathsf{Ch}^{\open}_{L \subset G} $, via $I \mapsto j^*_I: \Ch(U_{L,I}/'L_I) \to \Ch(U_{G,I} \cap \frl_I/'L_I),$ where $j_I: U_{G,I} \cap \frl_I/'L_I \to U_{L,I}/'L_I$ denote the open embedding.
 \end{itemize}
%$\sfRef$ is an natural equivalence, since each $j^*_I$ is an equivalence.
We have induced functors:
\begin{equation}
\xymatrix{ 	\lim_{\frF_G} \mathsf{Ch}^{\open}_G  \ar[r]^-{\beta^*}  &  \lim_{\frF_{L \subset G}} \mathsf{Ch}^{\open}_G \circ \beta \ar[r]^{\lim \sfRes_P}  &  \lim_{\frF_{L \subset G}} \mathsf{Ch}^{\open}_{L \subset G}   }
\end{equation}
\begin{equation}
	\xymatrix{ 	\lim_{\frF_L} \mathsf{Ch}^{\open}_L  \ar[r]^-{\alpha^*}  &  \lim_{\frF_{L \subset G}} \mathsf{Ch}^{\open}_L \circ \alpha \ar[r]^{\lim \sfRef}  &  \lim_{\frF_{L \subset G}} \mathsf{Ch}^{\open}_{L \subset G}   }
\end{equation}

\begin{thm} 
 In the above setting,
\begin{enumerate}
	\item There is a natural equivalence: 
	$$  \lim_{ I \in \frF_G } \sfCh^{\open}_G \simeq  \Ch(G/G)$$ 
	\item The functor  $$\alpha^*\circ \lim \sfRef : \lim_{ I \in \frF_{L \subset G} } \sfCh^{\open}_{L \subset G} \longrightarrow \lim_{ I \in \frF_{L} } \sfCh^{\open}_L$$
	is an equivalence.
\end{enumerate}	
Moreover, under the above equivalences, the diagram naturally commutes:
	\begin{equation}
		\label{eq:Ch_G_to_Ch_L}
		\xymatrix{   \lim_{ I \in \frF_G } \sfCh^{\open}_G  \ar[r]^-{\simeq} \ar[d]^{ \lim \sfRes_P \circ \beta^*} &      \Ch(G/G) \ar[d]^{\Res_P} \\
			\lim_{ I \in \frF_{L \subset G} }  \sfCh^{\open}_{L \subset G}   \ar[r]^-{\simeq} &  \Ch(L/L)	   
		}
	\end{equation}
\end{thm}

\begin{proof}
	(1) is a variation of \cite[Theorem 5.40]{liUniformizationSemistableBundles2021}, and can be proved similarly: each map $U_{G,I} \to G/G$ is \'etale (in the classical topology), 
	then the statement follows from Prop~\ref{prop:equiv_C_points} below. (2) The second equivalence is proved in (1). 
	For the first one: $\alpha^*$ is equivalence by Lemma~\ref{lem:fiber_alpha_contractible}.
	 Moreover, both $U_{G,I} \cap \frl_I$ and $U_{L_I}$ are star-sharped open subset of $\frl_I$ centered at some/any point in $I$. 
	 Therefore by \cite[Prop 4.8, Definition 4.5]{liUniformizationSemistableBundles2021}, $j_I^*$ is an equivalence, therefore $\lim \sfRef$ is also an equivalence. 
	Finally, the square commutes because over the local chart $U_{G,I} \cap \frl_I/'L_I$ of $ L/L$ and $U_{G,I}$ of $G/G$, the stack $P/P$ can identified as $U_{G,I} \cap \frp_I/P_I$, more precisely, the diagram commute:
	$$\xymatrix{  U_{G,I} \cap \frl_I/'L_I \ar[d]  &   U_{G,I} \cap \frp_I/'P_I \ar[r]^p  \ar[d] \ar[l]_q &  U_{G,I}/'G_I  \ar[d] \\
						L/L						&       P/P					 \ar[r]^p  \ar[l]_q			&    G/G
}$$
with both squares cartesian.
\end{proof}

\begin{lem}
	\label{prop:equiv_C_points}
The natural map $\colim_{\frF_G} U_{G,I}/'G_I(\mathbb{C})  \to G/G (\mathbb{C})$ is an equivalence of $\infty$-groupoid.
\end{lem}

\begin{proof}
It is clear that if the statement holds for $G_1$ and $G_2$, then it holds for $G_1 \times G_2$. Now let $Z \subset G$ be a finite central subgroup, 
put $H=G/Z$, and assume the statement holds for $G$. Then we have $U_{G,I} = U_{H,I}$, 
and $H/H = (G/G)/(Z/Z)= (\colim_{\frF_G} U_{G,I}/G_I)/(Z/Z) \simeq (\colim_{\frF_G} U_{H,I}/H_I)/Z = \colim_{\frF_H}  U_{H,I}/H_I,$ therefore the statement holds for $H$. 
Now suffices to show the statement holds for $\GG_m$ and simply-connected group. 
When $G=\GG_m$, $\frF_G= B\ZZ$, and $U_{G,\frt_\RR}= \CC$, therefore $\CC^*/\CC^* = (\CC/\CC^*)/\ZZ= \colim_{\frF_{\GG_m}} U_{\GG_m,\frt_\RR}/\GG_m$. 
When $G$ is simply-connected, this is \cite[Theorem 5.40 (6)]{liUniformizationSemistableBundles2021} and Remark~\ref{rmk:F_G} (2). 
\end{proof}

\subsubsection{Gluing functors on automorphic side}
\label{auto_graded}

\begin{lem}
	\label{prop:choice_of_s_I}
There exist choice $ s_I \in I$, for each $I$, such that
\begin{itemize}
	\item the collection of subsets $\{  S_I:=s_I + \frz_\frg  : I \in \frF_G \}$ is stable under $\tilW_G$-action;
	\item for all $I \subset \overline{J}$, $S_I \subset S_J + J^\perp$.
\end{itemize}
\end{lem}
\begin{proof}
Suffices to construct for the case when $G$ is simple and adjoint type. 
In this case, for alcove $A$ (which is a simplex), choose $s_A$ the binary-center of $A$, and for any $I \subset \overline{A}$, choose $s_I$ the orthogonal projection of $s_A$ in $I$. 
Easy to see this is well define: if $I \subset \overline{A} \cap \overline{A'}$, we can find alcoves $A=A_0, A_1, ..., A_n=A'$, 
such that $ \overline{A_i} \cap \overline{A_{i+1}}$ is codimension $1$ and contains $I$. 
Let $r$ be the reflection that send $A_{i}$ to $A_{i+1}$. 
Then the projections of $s_{A_i}$ and $s_{A_{i+1}}$ to $\overline{A_i} \cap \overline{A_{i+1}}$ are the same, hence projections to $I$ are also the same.
\end{proof}

Now for any facet $I$ in $\frt_\RR$, put $\overline{\frg}_I= S_I \otimes_\RR \CC + [\frg_I,\frg_I ] \subset \frg_I$ \footnote{For a affine subspace $A$ of a real vector space $V$,
 we put $A \otimes_\RR \CC$ be the unique complex affine subspace of $V \otimes_\RR \CC$ that contains $A$. }.  Then $\overline{\frg}_I$ is stable under twisted $G_I$-conjugation. 
  for $I \subset \overline{J}$. Moreover the $*$-restriction functors are equivalences:
\begin{equation}
	\label{eq:cs_open_and_closed}
	\xymatrix{ \Ch({\frg_I/'G_I})  \ar[d]   \ar[r]  & \Ch({U_{G,I}/'G_I} )   \ar[d] \\
		\Ch({\overline{\frg}_I/'G_I})  \ar[r]  &  \Ch( \overline{\frg}_I \cap U_{G,I} /'G_I ) }
\end{equation}

For $I \subset \overline{J}$, Prop~\ref{prop:choice_of_s_I} implies that $\overline{\frg}_J \subset \overline{\frg}_I$, and $\overline{\frg}_J \subset \overline{\frg}_I$, 
and $\overline{\frg}_{w(I)}= Ad_w(\overline{\frg}_I)$. 
The stack $\overline{\frg}_J/'G_I$ also satisfies the assumption in Appendix~\ref{secA:graded_lift} (because the $*$-restriction functor $\Ch_c(\frg_I/'G_I) \to \Ch_c(\overline{\frg}_J/'G_I)$ is an equivalence). 
Therefore we have its graded lift $\Ch_{\gr,c}(\overline{\frg}_I/'G_I)$. We shift the perverse $t$-structure on $\Ch_{\gr,c}(\overline{\frg}_I/'G_I)$ and $\Ch_{c}(\overline{\frg}_I/'G_I)$,
so that constant sheaves in degree 0 are in the $t$-heart (and hence the functor $\Ch_c(\frg_I/'G_I) \to \Ch_c(\overline{\frg}_J/'G_I)$ is $t$-exact).

\begin{defn}
	\label{defn：automorphic_functors} \begin{enumerate}
		\item Define the functors $\frF_G \to \Cat$:
			\begin{enumerate}
			\item $\sfCh_G : I \mapsto \Ch(\overline{\frg}_I /'G_I);$
			\item $\sfCh_{c,G} : I \mapsto \Ch_c(\overline{\frg}_I/'G_I);$
			\item $\sfCh^c_G : I \mapsto \Ch(\overline{\frg}_I/'G_I)^c;$
			\item $\sfCh_{\gr,c,G} : I \mapsto \Ch_{\gr,c}(\overline{\frg}_I/'G_I);$

		\end{enumerate}
		\item Define the functors $\frF_{L \subset G} \to \Cat$:
	\begin{enumerate}
		\item $\sfCh_{L \subset G} : I \mapsto \Ch(\overline{\frg}_I \cap {\frl}_I/'L_I);$
		\item $\sfCh_{c,L \subset G} : I \mapsto \Ch_c(\overline{\frg}_I \cap {\frl}_I/'L_I);$
		\item $\sfCh^c_{L \subset G} : I \mapsto \Ch(\overline{\frg}_I \cap {\frl}_I/'L_I)^c;$
		\item $\sfCh_{\gr,c,L \subset G} : I \mapsto \Ch_{\gr,c,G}(\overline{\frg}_I \cap {\frl}_I/'L_I);$

	\end{enumerate}
	\end{enumerate}

\end{defn}

Note that from (\ref{eq:cs_open_and_closed}) (and the similar argument for $\sfCh_{L \subset G})$, we have a commutative square of functors:
\begin{equation}
\xymatrix{  \sfCh^{\open}_G  \circ \beta   \ar[r]^{\simeq} \ar[d]^{\sfRes_P} &  \sfCh_G  \circ \beta \ar[d]^{\sfRes_P}  \\
	\sfCh^{\open}_{L \subset G} \ar[r]^{\simeq}    &  \sfCh_{L \subset G}  
}
\end{equation}
which induces functors between the limits:
\begin{equation}
\label{eq:Res_Ch_open_and_closed}	
	\xymatrix{  \lim \sfCh^{\open}_G     \ar[r]^{\simeq} \ar[d]^{\lim \sfRes_P \circ \beta} &  \lim \sfCh_G \ar[d]^{\lim \sfRes_P \circ \beta} \\
	 \lim 	\sfCh^{\open}_{L \subset G} \ar[r]^{\simeq}    & 
	 \lim  \sfCh_{L \subset G}  
	}
\end{equation}

\subsection{Spectral side}
\label{sec:spec_glue}
 In this section, we define the functors on the spectral side and calculate their limits.

\begin{defn}  \begin{enumerate}
	\label{defn:spectral_functors}	
		\item Introducing the following functors defined on $\frF_G$:
			\begin{enumerate}
			\item $\sfQCoh_{\tilfrC_G}^{\tilW_G} :   I \mapsto  (\prod_{c \in \frC_I} \Sym(\frz_c[1]) \modu)^{W_I}; $
			\item $\sfCoh_{\tilfrC_G}^{\tilW_G} :   I \mapsto  (\prod_{c \in \frC_I} \Sym(\frz_c[1]) \modfd)^{W_I}; $
			\item $\sfPerf_{\tilfrC_G}^{\tilW_G} :   I \mapsto  (\prod_{c \in \frC_I} \Sym(\frz_c[1]) \perf)^{W_I}; $
			\item $\sfCoh_{\gr,\tilfrC_G}^{\tilW_G} :   I \mapsto  (\prod_{c \in \frC_I} \Sym(\frz_c[1] \langle 2 \rangle) \modfd^\gr)^{W_I}; $
	%		\item $\sfCoh_{\gr,\tilfrC_G}^{\tilW_G, \heartsuit_{w,t}} :   I \mapsto  (\prod_{c \in \frC_I} \CC \modfd^\heartsuit)^{W_I}. $
		\end{enumerate}
	\item Introducing the following functors defined $\frF_{L \subset G} $:
		\begin{enumerate}
		\item $\sfQCoh_{\tilfrC_{L \subset G}}^{\tilW_L} :   I \mapsto  (\prod_{c \in \frC_{I_L}} \Sym(\frz_c[1]) \modu)^{W_{I_L}}; $
		\item $\sfCoh_{\tilfrC_{L \subset G}}^{\tilW_L} :   I \mapsto  (\prod_{c \in \frC_{I_L}} \Sym(\frz_c[1]) \modfd)^{W_{I_L}}; $
		\item $\sfPerf_{\tilfrC_{L \subset G}}^{\tilW_L} :   I \mapsto  (\prod_{c \in \frC_{I_L}} \Sym(\frz_c[1]) \perf)^{W_{I_L}}; $
		\item $\sfCoh_{\gr,\tilfrC_{L \subset G}}^{\tilW_L} :   I \mapsto  (\prod_{c \in \frC_{I_L}} \Sym(\frz_c[1] \langle 2 \rangle) \modfd^\gr)^{W_{I_L}}; $

	\end{enumerate}
	\end{enumerate}

\end{defn}

\begin{prop} 
 There are equivalences of dg categories: 
 \begin{enumerate}
 	\item $\lim_{ I \in \frF_G } \sfQCoh_{\tilfrC_G}^{\tilW_G} \simeq  \Gamma(\frt, \scrQCoh_{\widetilde{\frC}_G})^{\widetilde{W}_G} ;$
 	\item $ \lim_{ I \in \frF_{L \subset G} } \sfQCoh_{\tilfrC_{L \subset G}}^{\tilW_L} \simeq   \lim_{ I \in \frF_{L} } \sfQCoh_{\tilfrC_{L}}^{\tilW_L}   \simeq  \Gamma(\frt, \scrQCoh_{\widetilde{\frC}_L})^{\widetilde{W}_L}. $
 \end{enumerate}
 Moreover, The diagram naturally commutes:
 \begin{equation}
 	\label{eq:global_section_of_QCoh}
	\xymatrix{   \lim_{ I \in \frF_G } \sfQCoh_{\tilfrC_G}^{\tilW_G}  \ar[r]^{\simeq} \ar[d]^{(\lim \sfRes) \circ \beta^*} &      \Gamma(\frt, \scrQCoh_{\widetilde{\frC}_G})^{\widetilde{W}_G}  \ar[d]^{\Res} \\
	\lim_{ I \in \frF_{L \subset G} } \sfQCoh_{\tilfrC_{L \subset G}}^{\tilW_L}   \ar[r] &  \Gamma(\frt, \scrQCoh_{\widetilde{\frC}_L})^{\widetilde{W}_L}} 	   
 \end{equation}
 \end{prop}
\begin{proof}
(1) Note that by definition, we have $\scrQCoh_{\tilfrC_G} |_{V_I}   = \scrQCoh_{\frC_I} |_{V_I}$, therefore 
 $$\lim_{ I \in \frF_G } \sfQCoh_{\tilfrC_G}^{\tilW_G} \simeq \lim_{ I \in \frF_G } \Gamma(V_I, \scrQCoh_{{\frC}_I})^{{\tilW}_G}  \simeq  \lim_{ I \in \frF_G } \Gamma(V_I, \scrQCoh_{{\tilfrC}_G})^{{\tilW}_G}  \simeq  \Gamma(\frt, \scrQCoh_{{\tilfrC}_G})^{\tilW_G}$$ 
where the last equivalence follows from Lemma~\ref{prop:colim_of_V} below and Prop~\ref{prop:descent_for_constant_sheaf}.
For (2), the second equivalence is (1), and the first equivalence follows from Lemma~\ref{lem:fiber_alpha_contractible}. The commutativity of the diagram is straightforward.
\end{proof}

Define functor $\mathsf{V}:\frF_G \to \textup{Grpd}_\infty$, via $I \mapsto V_I/W_I,$ where $V_I$ are equipped with the discrete topology.

\begin{lem}
	\label{prop:colim_of_V}
	The natural maps $ \colim_{\frF_G} V_I/W_I \to \frt/\tilW $ is an isomorphism of $\infty$-groupoids.
\end{lem}
\begin{proof}
Similar to the proof of Prop~\ref{prop:equiv_C_points}, we reduce to the case when $G$ is simple and simply-connected or $G=\GG_m$. 
The former case is a well-known property for reflection group, see e.g. \cite[Prop 3.3]{liUniformizationSemistableBundles2021}, and the later case is an easy computation.
\end{proof}

\subsection{Matching the functors}
\label{sec:match} In this section, we match up the functors on the automorphic and spectral sides. We first start with the local identifications, which is an analog of Proposition~\ref{prop:identify_invariant_category} for the twisted action: 

\begin{prop}
	There is a canonical equivalence:
	\beq
	\label{eq:spectral_qcoh}
	\LL_{I}:  \Gamma(\frt, \scrQCoh_{\frC_I})^{W_I} \simeq \Ch({ \overline{\frg}_I/'G_I})  .
	\eeq 

Moreover, the diagram naturally commutes:
	\beq  
	\label{eq:identify_restriction_lie_algebra}
	\xymatrix{ \Gamma(\frt, \scrQCoh_{\frC_J})^{W_J}\ar[r]^-{\LL_{J}}  & \Ch_{\gr,c}(\overline{\frg}_J/'G) \\
		\Gamma(\frt, \scrQCoh_{\frC_I})^{W_I} \ar[r]^-{\LL_{I}} \ar[u]_{\Res} &  \Ch_{\gr,c}( \overline{\frg}_I/'G) \ar@<-.5ex>[u]_{j^*}
	}
	\eeq
\end{prop}

\begin{thm} 
\label{thm:identify_A_and_B}	
For $\bullet \in \{1,2\},$ and $\star \in \{a,b,c,d,e\}$,	denote by $\mathsf{A}_{\bullet, \star}$ and $\mathsf{B}_{\bullet, \star}$ the corresponding functors in Definition~\ref{defn：automorphic_functors} and ~\ref{defn:spectral_functors}, respectively.
\begin{enumerate}
	\item For any $\bullet$ and $\star$, there are natural isomorphism of functors: $\LL_{\bullet,\star} : \mathsf{A}_{\bullet, \star} \to \mathsf{B}_{\bullet, \star}.$
	\item For any $\star$, the diagram of functors naturally commutes: 
$$	\xymatrix{   \mathsf{A}_{1, \star} \circ \beta \ar[r]^{ \LL_{1,\star} \circ \beta } \ar[d]^{\Res_P}   &      \mathsf{B}_{1, \star} \circ \beta \ar[d]^{\Res}  \\
				 \mathsf{A}_{2, \star}   \ar[r]^{ \LL_{2,\star}}                  &      \mathsf{B}_{2, \star} 
}$$
\end{enumerate}
\end{thm}

\begin{proof}
	We prove (1) for the case $\bullet =1$, other statements can be proved similarly. We first prove the case $\star=d$. For any $I \to J $ in $\frF_G$,  \eqref{eq:identify_restriction_lie_algebra} gives commutative squares:

	$$\xymatrix{      \sfA_{1,d} (I)  = \Ch_{c,\gr}(\overline{\frg}_I/'G_I)  \ar[d] \ar[r]^-{\simeq} &     	
		  (\prod_{c \in \frC_I} \Sym(\frz_c[1] \langle 2 \rangle) \modfd^\gr)^{W_I} =	\sfB_{1,d} (I)	 \ar[d]^{\Res} 	 	 \\
	\sfA_{1,d} (J)  = \Ch_{c,\gr}(\overline{\frg}_J/'G_J) \ar[r]^-{\simeq}  &  
	  (\prod_{c \in \frC_J} \Sym(\frz_c[1] \langle 2 \rangle) \modfd^\gr)^{W_J}	 =	\sfB_{1,d} (J)			}$$
  These identifications are also compatible with compositions in $\frF_G$, and  gives  $\LL_{1,d}: \sfA_{1,d} \simeq  \sfB_{1,d}$. 
This induces $\LL_{1,b}= \LL_{1,d} \otimes_{\Vect^{\gr,c}} \Vect^c$, 
and since  $\LL_{1,b}(I)$ identifies the full subcategory 
$  \sfA_{1,c}(I) \subset \sfA_{1,b}(I)$ with  $\sfB_{1,c}(I) \subset \sfB_{1,b}(I)$, for all $I \in \frF_G$, 
therefore gives $\LL_{1,c}$. 
Finally, take $\LL_{1,a} = \Ind(\LL_{1,c})$.
\end{proof}

We have the main theorem of this note:

\begin{thm}
	\label{thm:main}
There is an equivalence of dg categories: 
$$ \LL_G: \Ch(G/G)  \simeq   (\prod_{c \in \widetilde{\frC}_G}  \Sym(\frz_c[1]) \textup{-mod})^{\widetilde{W}_G} . $$
Moreover, under the above identification, there are natural commutative squares:
$$\xymatrix{
	\Ch(L/L)  \ar@<-.5ex>[d]_{\Ind_{L \subset P}^G}  \ar[r]^-{\LL_L}  &    (\prod_{c \in \widetilde{\frC}_L}  \Sym(\frz_c[1]) \textup{-mod})^{\widetilde{W}_L}  \ar@<-.5ex>[d]_{\Ind} \\
	\Ch(G/G)  \ar[r]^-{\LL_G} \ar@<-.5ex>[u]_{\Res_{L\subset P}^G}  &  (\prod_{c \in \widetilde{\frC}_G}  \Sym(\frz_c[1]) \textup{-mod})^{\widetilde{W}_G} \ar@<-.5ex>[u]_{\Res} 
}$$
\end{thm}
\begin{proof}
We prove the statement for $\Res$, and the statement for $\Ind$ follows by adjunction.
Take the limit of the isomorphism in Theorem~\ref{thm:identify_A_and_B}, for $\star=a$. We have 
$$\xymatrix  {  \lim \sfCh_G  \ar[d]   \ar[r]^-{\simeq}       &  \lim \sfQCoh_{\tilfrC_G}^{\tilW_G} \ar[d]  \\  
	\lim \sfCh_G \circ \beta   \ar[r]^-{\simeq} \ar[d]      &  \lim \sfQCoh_{\tilfrC_{G}}^{\tilW_G}  \circ \beta  \ar[d] \\
 \lim \sfCh_{L \subset G}    \ar[r]^-{\simeq}       &  \lim \sfQCoh_{\tilfrC_{L \subset G}}^{\tilW_L}   
}  $$
Therefore we get commutative diagrams:
$$ \xymatrix{  \Ch(G/G) \ar[d]^{\Res_P} \ar[r]^{\simeq}   &    \lim \sfCh_G  \ar[d]^{ \lim \Res_P  \circ \beta^*}    \ar[r]   &   \lim \sfQCoh_{\tilfrC_G}^{\tilW_G} \ar[d]^{ \lim \Res  \circ \beta^*}   \ar[r] & \Gamma(\frt, \scrQCoh_{\widetilde{\frC}_G})^{\widetilde{W}_G} \ar[d]^{\Res}   \\
	 \Ch(L/L)  \ar[r]^{\simeq}   &    \lim \sfCh_{L \subset G}    \ar[r]   &   \lim \sfQCoh_{\tilfrC_{L \subset G}}^{\tilW_G}   \ar[r] & \Gamma(\frt, \scrQCoh_{\widetilde{\frC}_L})^{\widetilde{W}_L} }
$$
where the square on the left is the composition of (\ref{eq:Ch_G_to_Ch_L}) and (\ref{eq:Res_Ch_open_and_closed}), and the square on the right is (\ref{eq:global_section_of_QCoh}).

\end{proof}

Define the category of \textit{principal character sheaves} $\Ch(G/G)_{\prin} \subset \Ch(G/G)$, to be the full subcategory generated by colimits of the essential image of $\Ind_{T \subset B}^G$. 

\begin{cor}
	\label{cor:principal_character}
	The functors $\Ind_{L \subset P}^G $ and $\Res_{L \subset P}^G$ preserve the subcategories of principal character sheaves, and induce commutative diagrams:
	$$\xymatrix{
	\Ch(L/L)^{\prin}  \ar@<-.5ex>[d]_{\Ind_{L \subset P}^G}  \ar[r]^-{\simeq}  &    \CC[W_L] \#( \CC[X_*(T)] \otimes \Sym(\frt[1])) \textup{-mod} \ar@<-.5ex>[d]_{\Ind} \\
	\Ch(G/G)^{\prin}  \ar[r]^-{\simeq} \ar@<-.5ex>[u]_{\Res_{L\subset P}^G}  &  \CC[W_G] \#( \CC[X_*(T)] \otimes \Sym(\frt[1])) \textup{-mod} \ar@<-.5ex>[u]_{\Res} 
}$$
\end{cor}
\begin{proof}
It follows from the fact that  $\Ch(G/G)_{\prin}$ is identified with the summand of $c=(T,T,\CC_{0/T}) \in \tilfrC_G$ under Theorem~\ref{thm:main}.
\end{proof}

 Recall that a dg algebra $A$ is \textit{formal} if $A$ is isomorphic to $H^*(A)$ as dg-algebras, and we say a map $f:A \to B$ is \textit{formal} if $f$ is isomorphic to $H^*(f)$ as maps between dg-algebras (this, in particular, requires both $A$ and $B$ are formal). Using the explicit identification in Corollary~\ref{cor:principal_character}, we obtain:

\begin{cor} 
	\label{cor:formality}
	 Denote by $\Spr_G = \Ind_{T \subset B}^G \in \Ch(G/G)$ the Grothedieck-Springer sheaf.
	\begin{enumerate}
		\item \label{item:compute_alg_thm:li_transcendental_char} There is an equivalence of dg-algebras  
	\beq 
		\End(\Spr_G) \simeq \CC[W] \# (H^*(BT) \otimes H^*(T)). 
 \eeq
		In particular, $\End(\Spr_G)$ is a formal dg-algebra. 
	\item There is a natural commutative diagram of dg-algebra
\beq	\xymatrix{   \End(\CC_{T/T})  \ar[r]^{\ind} \ar[d]^{\simeq}   &    \End(\Spr_G)  \ar[d]^{\simeq} \\
				H^*(BT) \otimes H^*(T) 	\ar[r]    &      \CC[W] \# (H^*(BT) \otimes H^*(T))
	}	
\eeq	
where the bottom map is the natural inclusions. In particular, $\ind$ is a formal map.
		\item There is a natural equivalence $\Res_{T \subset B}^G (\Spr_G) \simeq \CC_{T/T}^{\oplus W}$. Moreover, the natural dg-algebra homomorphism $\End(\Spr_G) \to \End(\Res_{T \subset B}^G (\Spr_G)) \simeq \End( \CC_{T/T}^{\oplus W})$ can be expressed explicitly via the commutative diagram
\beq 
\xymatrix{
			\End(\Spr_G) \ar[r]^{\res} \ar[d]^{\simeq} & \End(\Res_{T\subset B}^G(\Spr_G)) \ar[d]^{\simeq} \\
			\CC[W] \# (H^*(BT) \otimes H^*(T)) \ar[r] \ar[d]^{\simeq}	& \End_{\CC}(\CC[W]) \otimes (H^*(BT) \otimes H^*(T)) \ar[d]^{\simeq} \\
			\End_{\CC[W] \# (H^*(BT) \otimes H^*(T))} (\CC[W] \# (H^*(BT) \otimes H^*(T))) \ar[r] & \End_{H^*(BT) \otimes H^*(T)} (\CC[W] \# (H^*(BT) \otimes H^*(T)))
}
\eeq
		where the bottom arrow is induced by the restriction of module structure along
		\[
	H^*(BT) \otimes H^*(T) \to \CC[W] \# (H^*(BT) \otimes H^*(T)).
		\]
	in particular, $\res$ is a formal map.
	\end{enumerate}
\end{cor}

\begin{rmk}
	For $G$ simply-connected, Corollary~\ref{cor:formality}(1) is a consequence of \cite[Corollary 1.9]{liDerivedCategoriesCharactera}.
\end{rmk}

\begin{cor}
	\label{cor:unipotent_character}
Denote by $\Ch^u_c(G/G)^{\prin} \subset \Ch(G/G)$ the small subcategory generated by finite colimits of $\Spr_G$, then we have commutative diagram:
	$$\xymatrix{
	\Ch^u_c(L/L)^{\prin}  \ar@<-.5ex>[d]_{\Ind_{L \subset P}^G}  \ar[r]^-{\simeq}  &    \CC[W_L] \# (H^*(BT) \otimes H^*(T)) \textup{-perf} \ar@<-.5ex>[d]_{\Ind} \\
	\Ch_c^u(G/G)^{\prin}  \ar[r]^-{\simeq} \ar@<-.5ex>[u]_{\Res_{L\subset P}^G}  &  \CC[W_G] \#(H^*(BT) \otimes H^*(T)) \textup{-perf} \ar@<-.5ex>[u]_{\Res} 
}$$
\end{cor}

\begin{eg}
	Under the equivalence in Corollary~\ref{cor:unipotent_character}, the constant sheaf $\CC_{G/G}$ corresponds to the perfect module $H^*(BT) \otimes H^*(T)$. This follows from the fact that in the local identifications \eqref{eq:spectral_gr}, the constant sheaves on $\frg/G$ corresponds to the $1$-dimensional sign module of $\CC[W] \# \Sym (\frt[1])$.
\end{eg}

\appendix

\section{Graded lift of categories}

We refer to \cite[Section 5]{hoRevisitingMixedGeometry2022} for the notion of graded lift/mixed version. %Note that in our situation, the equivalent conditions in \cite[Proposition 5.6.12]{hoRevisitingMixedGeometry2022} are assumed to be true.

\label{secA:graded_lift}

\subsection{Graded lift over $\overline{\mathbb{F}}_q$}
 Let $X$ be an algebraic stack over $\overline{\FF}_q$. Denote by $\Sh_c(X)$ the dg category of constructible $\overline{\QQ}_\ell$-sheaves on $X$. Let $S=\{A_i, i\in I\}$ be a set of irreducible perverse sheaves on $X$, assuming that each $A_i$ comes from some mixed sheaves on $X_n$, where $X_n$ is a ${\FF}_{q^n}$-form of $X$.  Then $\RHom_{\Sh(X)}(A_i,A_j)$ is naturally a graded complex of vector spaces, by considering Frobenius weight.
Assume the following condition holds:
\begin{equation}\label{eqn:purity}
	\RHom_{\Sh(X)}(A_i,A_j) \textup{ is pure.}  \qquad \forall A_i,A_j \in S
\end{equation}
Denote by $\Sh_{c,S}(X) \subset \Sh_c(X)$ full dg-subcategory generated by $S$. Put $\Pur_{c,S}(X) \subset \Sh_{c,S}(X)$ the additive 1-category of (shifts of) semisimple complexes. Define $\Sh_{\gr,c,S}(X)= K^b(\Pur_{c,S}(X) ) $, where $K^b(-)$ denotes the dg category of chain complexes in an additive category. Note that $\Pur_{c,S}(X) $ has an induced shift functor from  $\Sh_{c,S}(X)$. This can be viewed as an $\Vect^{\gr,c,\heartsuit_w}$ action on  $\Pur_{c,S}(X)$, which induced $\Vect^{\gr,c} \simeq  K^b(\Vect^{\gr,c,\heartsuit_w})$ action on $\Sh_{\gr,c,S}(X)$. We have the following by \cite[Proposition 5.6.12]{hoRevisitingMixedGeometry2022}:

\begin{prop}
	\label{prop：graded_lift_category}
	There is an canonical equivalence of dg categories: $$\Sh_{\gr,c,S}(X) \otimes_{\Vect^{\gr,c}}  \Vect^{c} \simeq \Sh_{c,S}(X).$$
\end{prop}

Let $F: \Sh_c(X) \to \Sh_c(Y)$ be a functor of geometric origin, i.e given by the composition of push and pull along maps of algebraic stacks. Assume that 
\begin{equation}
	\label{eqn:F preserve sheaves}
	F: \Sh_{c}(X) \to \Sh_{c}(Y) \textup{ send } S \textup{ into } T.
\end{equation}
\begin{equation}
	\label{eqn:F_n preserve purity}
	F_n: \Sh_{m,c}(X_n) \to \Sh_{m,c}(Y_n) \textup{ preserves pure objects of weight 0 for big } n.
\end{equation}

Where  $\Sh_{m,c}(X_n)$ denote the dg category of mixed constructible sheaves on $X_n$.  Put $F_{\Pur}= F|_{\Pur_{c,S}(X)} : \Pur_{c,S}(X) \to \Pur_{c,T}(Y)$. And $F_\gr:=K^b(F_{\Pur}): \Sh_{\gr,c,S}(X) \to \Sh_{\gr,c,T}(Y).$

\begin{prop}
	\label{prop：graded_lift_functor}
	Under the equivalences in Proposition~\ref{prop：graded_lift_category}, there is a canonical isomorphism $F_\gr \otimes Id \simeq F$.
\end{prop}

\subsection{Graded lift over $\CC$}
Now we assume $K$ is a number field, with the ring of integer $O$. And $X$ be an algebraic stack over $\Spec(O)$. Let $S=\{A_i,  i\in I \}$ be a finite set of sheaves of geometric origin (they are summands of push-pull of constant sheaves along algebraic maps). Assume that $ A_i \otimes_O \overline{O/\frp}$ is irreducible perverse sheaf for all but finitely many prime $\frp$.  

Then after removing finitely many primes $\frp_1,...\frp_r$, each  $A_{i}$ is locally constant along $\Spec(O_{\frp_1\frp_2...\frp_r})$. 
This gives an natural identification $\Sh_{c,S_\CC}(X_\CC) \simeq \Sh_{c,S_\frp}(X_{\frp})$, and $\Pur_{c,S_\CC}(X_\CC) \simeq \Pur_{c,S_\frp}(X_{\frp})$, with the evident notation. Similarly as before, we put $\Sh_{\gr,c,S_\CC}(X_\CC)= K^b(\Pur_{c,S_\CC}(X_\CC)).$

Suppose $S_\frp$ satisfies (\ref{eqn:purity}), then we have :
\begin{equation} 
	\label{eqn:graded_lift_category_over_C}
	\Sh_{\gr,c,S_\CC}(X_\CC) \otimes_{\Vect^{\gr,c}}  \Vect^{c} \simeq \Sh_{c,S_\CC}(X_\CC).
\end{equation}

Moreover, let $F_\CC: \Sh_c(X_\CC) \to \Sh_c(Y_\CC)$ be a functor of geometric origin, i.e given by the composition of push and pull along maps of algebraic stacks over $O$. 
Assume (\ref{eqn:F preserve sheaves}), (\ref{eqn:F_n preserve purity}) holds for $F_{\frp}$. Then there is a canonical isomorphism of functors:
\begin{equation} 
	F_{\CC,\gr} \otimes Id := K^b(F_{\CC,\Pur})  \otimes Id  \simeq F_\CC.
\end{equation}

\section{Sheaves and descent}
\label{secA:sheaves_of_categories}

Let $X$ be a topological space, $\mathscr{C}$ an $\infty$-category with arbitrary limits.

\begin{defn}
	A \textit{sheaf on $X$ with valued in} $\calC$ is a functor $F: \textup{Open}(X)^{op} \to \mathscr{C}$, such that for any $U \in \textup{Open}(X)$, and $U= \cup_{i \in I} U_i$ an open cover of $U$.
	The natural functor:
	$$ F(U) \to \lim_{n \in \Delta} \prod_{i_1,i_2,...i_n \in I} F(U_{i_1} \cap...\cap U_{i_n})  $$ 
	is an isomorphism in $\mathscr{C}$.
\end{defn}

Denote by $\Sh(X,\mathscr{C})$ the $\infty$-category of sheaves on $X$ with valued in $\mathscr{C}$.

\begin{eg}[Constant sheaf]
	Let $c \in \mathscr{C}$ be an object, then we have the constant sheaf $\underline{c}_X$ on $X$, defined by $\underline{c}_X(U) = \lim_U c$, where the right hand side means we view $U$ as an $\infty$-groupoid, and then take the limit of the constant functor mapping from $U$ to the single object $c$.
\end{eg}

\begin{defn}
	Let $K$ be a discrete group acting properly discontinuously on $X$. A \textit{$K$-equivariant sheaf} on $X$ is an object in the invariant category $\Sh(X,\mathscr{C})^K$.
	Concretely, a $K$-equivariant sheaf is the data of: 
	\begin{itemize}
		\item a sheaf $F$ on $X$,
		\item an isomorphism $\varphi_w: w^*F \xrightarrow{\sim} F,$ for each $w \in W$,
		\item an homotopy equivalence $\varphi_{u,v}: v^*(\varphi_u) \circ \varphi_v \xrightarrow{\sim} \varphi_{uv},$
		\item and higher compatibilities...
	\end{itemize}
\end{defn}

Define the $(2,1)$-category $\textup{Sub}(X,K)$ with:
\begin{itemize}
	\item Objects: $(U,W)$, with $U$ an open subset of $X$, and $W \subset K$ a subgroup, such that $U$ is stable under $W$-action.
	\item 1-morphisms: $w:(U_1,W_1) \to (U_2,W_2)$, for $ w \in K$, such that $w(U_1) \subset U_2$ and $Ad_w(W_1) \subset W_2$.
	\item 2-morphisms are inner homomorphisms: $u: w_1 \Rightarrow uw_1: (U_1,W_1) \to (U_2,W_2)$, for any $u \in W_1$. 
\end{itemize}	

\begin{prop}
	\label{prop:descent_for_constant_sheaf}
Let $(\mathsf{U},\mathsf{W}): \mathscr{I} \to \textup{Sub}(X,K)$ be a functor, sending $i \to (U_i,W_i)$. Assume that the induced map is an equivalence in the $\infty$-category of $\infty$-groupoid:
$$    \colim_{\mathscr{I}} U_i/W_i  \longrightarrow  X/\Gamma $$ 
where $U_i$ and $X$ are equipped with \textit{discrete} topology.
Then for any $K$-equivariant sheaf $F$ on $X$, which is a product of constant sheaves. There is an equivalence:
$$ \xymatrix{ \Gamma(X, F)^K  \ar[r]^-{\simeq} &  \lim_{\mathscr{I}}  \Gamma(U_i,F)^{W_i}   }$$
\end{prop}
\begin{proof}
Suppose $F= \prod_{i \in I} \underline{{c_i}}_{Z_i}$, for some $c_i \in \mathscr{C}$, and subspace $Z_i \subset X$. Put the functor $\mathsf{F}: \coprod \Pi(Z_i) \to \mathscr{C}^{op}$, sending $\Pi(Z_i)$ constantly to $c_i$, where $\Pi(Z_i)$ is the underlying $\infty$-groupoid of $Z_i$. Then by definition $ \colim \mathsf{F} \simeq \lim \mathsf{F} \simeq  \Gamma(X,F)$. Now the $K$-equivariant structure upgrade the functor to a functor $\mathsf{F}/K: (\coprod \Pi(Z_i)) /K \to \mathscr{C}^{op}$, with  $\colim \mathsf{F}/K \simeq \Gamma(X,F)^K$. Now by the $\infty$-categorical Van Kampen theorem, there is an equivalence of $\infty$-groupoid $\colim_{\mathscr{I}}\Pi(U_i)/W_i \simeq \Pi(X)/K$. This induces an isomorphism in $\mathscr{C}^{op}$:   $\Gamma(X,F)^K  \simeq   \colim_{\Pi(X)/K} \mathsf{F}/K \simeq  \colim_{\mathscr{I}} \colim_{\Pi(U_i)/W_i } \mathsf{F}/K \simeq \colim_{\mathscr{I}} \Gamma(U_i,F)^{W_i} . $
\end{proof}

\bibliographystyle{alpha}
\bibliography{char_shv_II}

\begin{thebibliography}{Bonb}

\bibitem[Bona]{bondarkoWeightStructuresWeights2012}
Mikhail~V. Bondarko.
\newblock Weight structures and “weights” on the hearts of
  \$t\$-structures.
\newblock 14(1):239--261.

\bibitem[Bonb]{bondarkoWeightStructuresVs2010}
M.V. Bondarko.
\newblock Weight structures vs. {\emph{t}} -structures; weight filtrations,
  spectral sequences, and complexes (for motives and in general).
\newblock 6(3):387--504.

\bibitem[ES]{elmantoNilpotentExtensionsInfty2021}
Elden Elmanto and Vladimir Sosnilo.
\newblock On nilpotent extensions of \$\textbackslash infty\$-categories and
  the cyclotomic trace.

\bibitem[Gin]{ginzburgInductionRestrictionCharacter1993}
Victor Ginzburg.
\newblock Induction and restriction of character sheaves.
\newblock 16:149--167.

\bibitem[HL]{hoRevisitingMixedGeometry2022}
Quoc~P. Ho and Penghui Li.
\newblock Revisiting mixed geometry.

\bibitem[JY]{johnson2DimensionalCategories2020}
Niles Johnson and Donald Yau.
\newblock 2-{{Dimensional Categories}}.

\bibitem[Li]{liDerivedCategoriesCharactera}
Penghui Li.
\newblock Derived categories of character sheaves.

\bibitem[LL]{laumonDerivedLusztigCorrespondence2023}
Gérard Laumon and Emmanuel Letellier.
\newblock On the derived {{Lusztig}} correspondence.
\newblock 11:e15.

\bibitem[LN]{liUniformizationSemistableBundles2021}
Penghui Li and David Nadler.
\newblock Uniformization of semistable bundles on elliptic curves.
\newblock 380:107572.

\bibitem[Lura]{lurieHigherAlgebra2012}
Jacob Lurie.
\newblock {\em Higher Algebra}.

\bibitem[Lurb]{lurieHigherToposTheory2009a}
Jacob Lurie.
\newblock {\em Higher Topos Theory}.
\newblock Number 170 in Annals of Mathematics Studies. {Princeton University
  Press}.

\bibitem[Lusa]{lusztigCuspidalLocalSystems1988}
George Lusztig.
\newblock Cuspidal local systems and graded {{Hecke}} algebras, {{I}}.
\newblock 67(1):145--202.

\bibitem[Lusb]{lusztigFourierTransformsSemisimple1987a}
George Lusztig.
\newblock Fourier transforms on a semisimple {{Lie}} algebra over
  {{F}}{\textsubscript{q}}.
\newblock In Arjeh~M. Cohen, Wim~H. Hesselink, prefix=van der~useprefix=true
  family=Kallen, given=Wilberd L.~J., and Jan~R. Strooker, editors, {\em
  Algebraic {{Groups Utrecht}} 1986}, volume 1271, pages 177--188. {Springer
  Berlin Heidelberg}.

\bibitem[Pau]{pauksztelloNoteCompactlyGenerated2011}
David Pauksztello.
\newblock A note on compactly generated co-t-structures.

\bibitem[RR]{riderFormalityLusztigGeneralized2021}
Laura Rider and Amber Russell.
\newblock Formality and {{Lusztig}}’s {{Generalized Springer
  Correspondence}}.
\newblock 24(3):699--714.

\end{thebibliography}

\end{document}